\documentclass{amsart}
\usepackage[utf8]{inputenc}

\usepackage{amsmath}
\usepackage{amssymb}
\usepackage{enumitem}
\usepackage{mathrsfs}
\usepackage{stmaryrd} % for double brackets

\usepackage[colorlinks]{hyperref}

\usepackage[textsize=footnotesize,textwidth=15ex]{todonotes}

\usepackage{xcolor}

\definecolor{green}{rgb}{0,.5,0} 

\newcommand{\Aut}{\mathrm{Aut}}
\newcommand{\Hom}{\mathrm{Hom}}

\newcommand{\Sym}{\mathrm{Sym}}
\newcommand{\Alt}{\mathrm{Alt}}
\newcommand{\SL}{\mathrm{SL}}
\newcommand{\Sp}{\mathrm{Sp}}
\newcommand{\EL}{\mathrm{EL}}
\newcommand{\id}{\mathrm{id}}

\newcommand{\F}{\mathbf{F}}
\newcommand{\Z}{\mathbf{Z}}
\newcommand{\Q}{\mathbf{Q}}

\newcommand{\trans}[4]{{\alpha_{#1;#2}^{(#3)}(#4)}}
\newcommand{\dtrans}[6]{{\alpha_{#1;#2,#3}^{(#4,#5)}(#6)}}
\newcommand{\Trans}[3]{{A_{#1;#2}^{(#3)}}}
\newcommand{\dTrans}[5]{{A_{#1;#2,#3}^{(#4,#5)}}}

\theoremstyle{plain} %% This is the default
\newtheorem{thm}{Theorem}[section]
\newtheorem{lem}[thm]{Lemma}
\newtheorem{prop}[thm]{Proposition}

\newtheorem{lemma}[thm]{Lemma}

\newtheorem{coro}[thm]{Corollary}
\newtheorem{step}{Step}

\theoremstyle{definition}
\newtheorem{remark}[thm]{Remark}
\newtheorem{observation}[thm]{Observation}
\newtheorem{definition}[thm]{Definition}

\title[Property (T) groups with many alternating quotients]{Tame automorphism groups of polynomial rings with property (T) and infinitely many alternating group quotients}

\author[P.-E.~Caprace]{Pierre-Emmanuel Caprace}
\address{Institut de Recherche en Math\'ematiques et Physique, UCLouvain, Belgium}
\thanks{PEC is supported in part by the FWO and the F.R.S.-FNRS under the EOS programme (project ID 40007542)} 

\author[M.~Kassabov]{Martin Kassabov}
\address{Cornell University, USA}
\thanks{MK is supported in part by Simons Foundation Grant 713557}

\date{May 8, 2023}

\begin{document}

\begin{abstract}
We construct  new families of groups with property (T) and infinitely many alternating group quotients.  One of those consists of  subgroups of $\Aut(\F_{p}[x_1, \dots, x_n])$ generated by a suitable set of tame automorphisms. Finite quotients are constructed using the natural action of $\Aut(\F_{p}[x_1, \dots, x_n])$ on the $n$-dimensional affine spaces over finite extensions of $\F_p$. As a consequence, we obtain explicit presentations of Gromov hyperbolic groups with property~(T) and infinitely many alternating group quotients. Our construction also yields an explicit infinite family of expander Cayley graphs of degree~$4$ for alternating groups of degree $p^7-1$ for any odd prime $p$. 
\end{abstract}

\maketitle

\setcounter{tocdepth}{1}
\tableofcontents

\section{Introduction}

Recent works by M.~Kaluba, P.~Nowak, N.~Ozawa~\cite{KNO},  M.~Kaluba, D.~Kielak, P.~Nowak~\cite{KKN} and M. Nitsche~\cite{N} showed that the automorphism group $\Aut(F_n)$ of the free group of rank~$n$ has Kazhdan's property (T) for all $n \geq 4$. 
Earlier, R.~Gilman~\cite{Gilman} had proved that $\Aut(F_n)$ has infinitely many alternating group quotients for all $n \geq 3$. To the best of our knowledge, this is up to now the only known family of group with property (T) and with infinitely many alternating  group quotients. The existence of families of groups with property ($\tau$) (which is weaker variant of (T)) and infinitely many alternating  group quotients follows from the work of the second author~\cite{K2}. 
The goal of this paper is to provide an explicit construction of other families of groups with property (T) and with infinitely many alternating  group quotients. As an application of our methods, we also provide an explicit construction of a triple of permutations which generate $\Alt(p^3-1)$ for any odd prime $p$, such that the resulting Cayley graphs are expanders, and another explicit pair of permutations in $\Alt(p^7-1)$ with the same properties.

\medskip

\subsection{Kazhdan groups with infinitely many alternating group quotients}

To describe our construction, we fix a prime $p$ and  an integer $n \geq 3$. We consider the commutative ring $\F_p[x_1, \dots , x_n]$ of polynomials in $n$ indeterminates with coefficients in the field $\mathbf F_p$ of order~$p$. For each integer $e \geq 1$ and all $i=1, \dots, n$, we consider the automorphism $\tau_i^{(e)} \in \Aut(\F_p[x_1, \dots , x_n])$ defined by the assignments
$$\tau_i^{(e)} \colon x_\ell \mapsto \left\{\begin{array}{ll}
x_i + (x_{i+1})^e & \text{if } \ell = i < n\\
x_n + (x_1)^e & \text{if } \ell = i = n\\
x_\ell & \text{if } \ell \neq i
\end{array}\right.$$
Those automorphisms can be thought of as \textit{polynomial transvections}. 

For every $n$-tuple $\mathbf{e}= (e_1, \dots, e_n)$ of strictly positive integers, we associate the group 
$$
G_{\F_p, \mathbf{e}} = \langle \tau_1^{(e_1)}, \dots, \tau_n^{(e_n)}\rangle \leq \Aut(\F_p[x_1, \dots, x_n]).
$$
The number $n$ is called the \textbf{rank} of $G_{\F_p, \mathbf{e}}$.
By definition the group $G_{\F_p, \mathbf{e}}$ %$G_{\F_p, n; e_1, \dots, e_n}$ 
consists of \textbf{tame automorphisms} of the ring $\F_p[x_1, \dots, x_n]$ (see~\cite{ShestakovUmirbaev} for more details on background and terminology on automorphism groups of polynomial rings).
It is easy to see that the rank~$n$ group $G_{\F_p, (1, 1, \dots, 1)}$ is isomorphic to $\SL_n(\F_p)$. The isomorphism maps the generator $\tau_i^{(1)}$ on the unipotent matrix $1+E_{i+1, i}$ for $i < n$ (resp. $1+E_{1, n}$ if $i=n$). In particular $G_{\F_p, (1, 1, \dots, 1)}$ is a quasi-simple finite group. However, as soon as $\max \{e_1, \dots, e_n\} \geq 2$, the group $G_{\F_p, \mathbf{e}}$ is residually a finite $p$-group, see Proposition~\ref{prop:Res-p} below.
Using criteria due to Ershov--Jaikin~\cite{EJ} and Kassabov~\cite{K}, we shall prove the following. 

\begin{thm}\label{thm:T-intro}
Let $n \geq 3$ and $\mathbf{e}= (e_1, \dots, e_n)$ be a tuple of positive integers. For each prime $p$ such that $p > 4\max\{e_1, \dots, e_n\}$,  the group $G_{\F_p, \mathbf{e}}$ has Kazhdan's property (T).
\end{thm}

We refer to Section~\ref{sec:Kazhdan} for more general results as well as estimates of the Kazhdan constants. The bound for $p$ in Theorem~\ref{thm:T-intro}
is not the best possible and can be slightly improved when the exponents $e_i$ are not all equal.
It should be underlined  that when $p > \max\{e_1, \dots, e_n\}> 1$, 
the group $G_{\F_p, \mathbf{e}}$ is infinite; indeed it contains an infinite elementary abelian $p$-group (see Corollary~\ref{cor:ElementaryAbelian} below). 

Our next goal is to construct a  family of large finite quotients of those groups. The source of such finite quotients is provided by the natural action of the group 
$\Aut(\F_p[x_1, \dots, x_n])$ on the points of the affine space $k^n$, where $k$ is a finite field of characteristic~$p$: a point $(a_1, \dots, a_n) \in k^n$  may be viewed as a homomorphism of $\F_p$-algebras  $\F_p[x_1, \dots, x_n] \to k : f \mapsto f(a_1, \dots, a_n)$. By pre-composing such a homomorphism by an element of $\Aut(\F_p[x_1, \dots, x_n])$, we obtain a natural permutation action of $\Aut(\F_p[x_1, \dots, x_n])$ on the finite set $k^n$. In particular we obtain a homomorphism
$$
G_{\F_p,\mathbf{e}} \to \Sym(k^n),
$$
which  takes its values in $\Alt(k^n)$ if $p> 2$, since the group is generated by elements of order $p$.  This permutation action is not transitive: it fixes the origin $(0, \dots, 0)$ and preserves the subset $F^n$ for each intermediate field $\F_p \subset F \subset k$. Moreover, the action of $G_{\F_p,\mathbf{e}}$ commutes with the natural action of the Galois group $\Aut(k)$, which is generated by the Frobenius automorphism. Under suitable assumptions on the parameters, we will show that $G_{\F_p,\mathbf{e}}$ acts $t$-transitively on a suitable quotient of a large subset of $k^n$, with $t \geq 4$. It is a well known consequence of the Classification of the Finite Simple Groups (CFSG) that a finite $4$-transitive group on a set of cardinality
~$\geq 25$ is the full alternating or symmetric group on that set (see~\cite[Th.~4.11]{Cameron}). The following result will be derived. 

\begin{thm}\label{thm:Alt-quotient-intro}
Let $n \geq 3$ and $\mathbf{e} = (e_1, \dots, e_n)$ be a tuple of positive integers. Suppose that $E= e_1\dots e_n \geq 2$. For each prime  $p \geq 3E-2$, the group $G_{\F_p,\mathbf{e}}$ has a quotient isomorphic to $\Alt(d)$ for infinitely many degrees~$d$.
\end{thm}

By considering suitable special cases, we shall obtain the following result, where the degrees of the alternating group quotients are explicit (the extra assumptions in part (ii) are needed to ensure that the alternating quotient is of the correct size).

\begin{coro}\label{coro:Alt-intro}
\ %Let $p$ be a prime. 
\begin{enumerate}[label=(\roman*)]
    \item  Let  $p \geq 5$ be a prime and $n\geq 3$ an integer. Then rank~$n$ group $G_{\F_p,( 1, 1, \dots, 1, 2)}$ has a quotient isomorphic to $\Alt(\frac{p^{n\ell}-p^n}\ell)$ for each prime $\ell \geq 3$. 
    
    \item Let $\sigma \in \Aut(\F_p[x_1, x_2, x_3])$ be the automorphism defined by  $\sigma(x_i) = x_{i+1}$, where indices are taken modulo~$3$. Let $\widetilde G = \langle \sigma \rangle \ltimes G_{\F_p,(2, 2, 2)}$. If $p \geq 23$ and $p \neq 1 \mod 7$, then the group  $\widetilde G$ has a quotient isomorphic to $\Alt(\frac{p^{3\ell}-p^3}\ell)$ for each prime $\ell \geq 5$.
\end{enumerate}
\end{coro}

The dependence on the CFSG can actually be removed --- indeed,  the proof of Theorem~\ref{thm:Alt-quotient-intro} relies on Theorem~\ref{thm:almost-k-transitivity}, which provides a $t$-transitive action of the group $G_{\F_p,\mathbf{e}}$  on a finite set set of size $d$, with $t \approx d^{1/n}$. In view of results of L.~Pyber~\cite{Pyber} (or earlier results of L.~Babai~\cite{Babai}), that do not depend on the CFSG, it follows that the image of the group contains the alternating group $\Alt(d)$ as soon as $d$ is large enough.

\subsection{Hyperbolic Kazhdan groups}\label{sec:Hyp-Kazhdan}

As a consequence of Theorems~\ref{thm:T-intro} and~\ref{thm:Alt-quotient-intro}, we obtain examples of hyperbolic groups with Kazhdan's property (T) admitting infinitely many alternating group quotients.  Indeed, for $1\leq e_1, e_2, e_3\leq 2$, the group $G_{\F_p,(e_1, e_2, e_3)}$ is a quotient of a generalized triangle group considered in~\cite{CCKW}, and called a \textbf{KMS group}. 
If $\max\{e_1, e_2, e_3\} = 2$,
that KMS group is hyperbolic and has property (T) if $p \geq 11$ (see~\cite[Sec.~7]{CCKW}). As a consequence of Theorem~\ref{thm:Alt-quotient-intro}, it has infinitely many alternating group quotients as soon as $p$ is sufficiently large. In particular, 
we obtain a positive  answer to the first part of~\cite[Question~1.6]{CCKW}. As a more specific illustration, one can observe that the KMS group%
\footnote{Here $[g,h]=g^{-1} h^{-1} gh$ is the group commutator, and $\llbracket g,h_1,h_2]$, $\llbracket g,h_1,h_2,h_3]$ denote the left normed  commutators, i.e.,  $[[g,h_1],h_2]$, $[[[g,h_1],h_2],h_3]$. }
\begin{align*}
\mathscr G_{HC_2^{(1)}}(p) = \langle a, b, c  \mid  & a^p, b^p, c^p, \llbracket a, b, a], \llbracket a, b, b], \\
& \llbracket b, c, b], \llbracket b, c, c],  \llbracket a, c, a], \llbracket a, c, c, a], \llbracket a, c, c, c]\rangle,
\end{align*}
maps onto $G_{\F_p,(1, 1, 2)}$ by sending the triple $(a, b, c)$ to $( \tau_3^{(2)}, \tau_2^{(1)}, \tau_1^{(1)})$, see Corollary~\ref{cor:nilpotent} below. Similarly, the group
$$\widetilde{\mathscr G}_{HBC_2^{(3)}}(p) = \langle t, a, b    \mid  t^3, a^p, tat^{-1}b^{-1},  \llbracket a, b, a], \llbracket a, b, b, a], \llbracket a, b, b, b ]\rangle $$
naturally maps onto the group $\widetilde G = \langle \sigma \rangle \ltimes G_{\F_p,(2, 2, 2)}$ from Corollary~\ref{coro:Alt-intro}, by sending $(t, a, b)$ to $(\sigma, \tau_1^{(2)}, \tau_2^{(2)})$. (The notation is borrowed from \cite{CCKW}.) By \cite[Theorem~1.3]{CCKW}, those two finitely presented groups are infinite hyperbolic as soon as $p$ is an odd prime; they have property (T) if $p \geq 7$ (resp. $p \geq 11$).  In view of Corollary~\ref{coro:Alt-intro}, we deduce the following. 

\begin{coro}\label{cor:hyp}
%There exists a Gromov hyperbolic group with property (T) and with infinitely many alternating group quotients. 
If $p \geq 5$ be a prime, the hyperbolic groups $\mathscr G_{HC_2^{(1)}}(p)$ maps onto $\Alt(\frac{p^{3\ell}-p^3}\ell)$ for each prime $\ell \geq 3$. 

If $p \geq 23$ and $p \neq 1 \mod 7$, then $\widetilde{\mathscr G}_{HBC_2^{(3)}}(p)$ maps onto $\Alt(\frac{p^{3\ell}-p^3}\ell)$ for each prime $\ell \geq 5$. 
\end{coro}

To our knowledge, these are the first explicit presentations of hyperbolic Kazhdan groups with infinitely many alternating group quotients. It should be noted that, as a consequence of \cite[Cor.~1.2]{BelegOsin}, every finitely presented Kazhdan group is a quotient of a hyperbolic Kazhdan group. Therefore, the very existence of hyperbolic Kazdhan groups with infinitely many alternating quotients  can be established using the fact that  $\mathrm{Aut}(F_n)$ is finitely presented, has (T) for $n \geq 4$, and  has infinitely many alternating group quotients. 
%The interest of Corollary~\ref{cor:hyp} is that it provides explicit presentations of such groups. 

\subsection{Alternating groups as expanders}

We recall that, given a infinite group $G$ with Kazhdan's property (T), Cayley graphs of  finite quotients of $G$ naturally form  expander graphs  (see \cite[Theorem~6.1.8]{BHV}). In particular, the results mentioned above yield examples of families of expander graphs arising as Cayley graphs for suitable generating sets of alternating groups. For example, the groups in Corollary~\ref{coro:Alt-intro} are respectively $n$-generated and $2$-generated, so we obtain expander Cayley graphs of degree~$2n$ (for any $n \geq 3$) and~$4$ for an infinite family of alternating groups. 

The flexibility of our construction allows us to provide  such Cayley graphs for a larger collection of alternating groups.  Indeed, we shall consider a more general family of groups defined as analogues of the groups mentioned so far, constructed as subgroups of $\Aut(R[x_1, \dots, x_n])$ generated by polynomial transvections, where $R$ is an arbitrary commutative unital ring. In particular,  we shall construction a subgroup $G$ of $\Aut(\Z[1/30][x_1, \dots, x_n])$ with property (T) (see Theorem~\ref{thm:T-bis}). Using the functoriality properties of the construction, we show that for each prime $p \geq 7$, the group $G$ acts by permutations on a set of cardinality $p^n-1$, so that the image contains the natural image of $G_{\F_p, (1, 1, \dots, 1, 2)}$, which is $\Alt(p^n-1)$. This will lead us to the following construction of expanding generating sets for the family of alternating groups $\Alt(p^3-1)$  (resp. $\Alt(p^7-1)$) indexed by the prime $p$.

\begin{thm}\label{thm:expanders-intro}
Let $p$  be an odd prime prime. 

\begin{enumerate}[label=(\roman*)]
    \item The permutations 
$$
\sigma(x,y,z) = (y,z,x)
\quad
\alpha(x,y,z) = (x+y,y,z)
\quad
\beta(x,y,z) = (x+y^2,y,z),
$$
acting on the set $\F_p^3 \setminus \{(0, 0, 0)\}$ of cardinality $p^3-1$, generate the full alternating group $\Alt(p^3-1)$. The associated Cayley graphs form expanders of degree~$6$. 

\item The permutations 
$$
\rho(x_1,x_2, x_3, x_4, x_5, x_6, x_7) = (x_2, x_3, x_4, x_5, x_6, x_7,x_1)
$$
$$
\gamma(x_1,x_2, x_3, x_4, x_5, x_6, x_7) = (x_1 + x_2 ,x_2, x_3, x_4+x_6^2, x_5, x_6, x_7)
$$
acting on the set $\F_p^7 \setminus \{(0, \dots, 0)\}$ of cardinality $p^7-1$, generate the full alternating group $\Alt(p^7-1)$. The associated Cayley graphs form expanders of degree~$4$. 
\end{enumerate}
\end{thm}

Although the proof of that result relies on the methods introduced to prove Theorems~\ref{thm:T-intro} and~\ref{thm:Alt-quotient-intro},  we do not know if the groups generated by $\sigma, \alpha, \beta$ (resp. $\rho, \gamma$) have a common cover with Kazhdan's property~(T). 

Let  
us briefly compare those results with the corresponding results for $\Aut(F_n)$ mentioned above. In the case of  $\Aut(F_n)$, the existence of infinitely many alternating group quotients is  established by a direct argument in~\cite{Gilman}; it does not rely on the CFSG (which was actually not available at the time). 
The fact that $\Aut(F_n)$ has property (T) for $n \geq 4$ is a recent achievement building upon Ozawa's criterion~\cite{O}, that relies on a fair amount of computer calculations, see~\cite{KNO, KKN, N}. For the groups   $G_{\F_p,( e_1, e_2, \dots, e_n)}$ considered in this paper, the proof of property (T) is a rather straightforward consequence of known criteria from~\cite{EJ, K}. The construction of finite quotients that are realized as multiply transitive permutation groups (of arbitrarily large degree) is elementary and self-contained, though it is quite technical at some places.

This family of expander Cayley graphs for alternating groups cover a  ``denser'' set of degrees, compared to the expanders obtained by using the fact that  $\Aut(F_4)$ has property (T). Indeed, the latter group  has a family of alternating  group quotients in degree approximately equal to $p^{12}$, where  $p$ is prime.

It should be noted that once one constructs  expanding generating sets  (of bounded size) for a family  $\Alt(n_k)$ of alternating groups, where $n_k$ grows not faster than exponentially, it is possible to modify the construction to obtain expanding generating sets  (of bounded size) for all alternating groups.

\subsection*{Acknowledgements}

We thank Jack Button and Piotr Przytycki for their comments on a preliminary version of this paper.  We are grateful to Greg Kuperberg for pointing out a similarity between the
expanding generating sets from Theorem~\ref{thm:expanders-intro} and the Toffoli gate in reversible computing. We also  thank the anonymous referee for numerous suggestions which significantly improved the  the paper.

\section{Preliminaries}

\subsection{A family of nilpotent groups}
\label{subsec:nilpotentgroups}

Let $R$ be a 
commutative ring% 
\footnote{All rings considered in the paper are associative with a unit.}
with a unit, and $c \geq 0$ be an integer. We let $R[x]$ be the $R$-module consisting of the polynomials in the indeterminate $x$ with coefficients in $R$, and $R[x]_{\leq c} \cong R^{c+1}$ be the submodule consisting of the polynomials of degree at most~$c$. 

For each $r \in R$, we consider the automorphism $y(r) : R[x] \to R[x]$ acting by a change of coordinates, i.e., $y(r)\big(P(x)\big) = P(x+r)$ . Clearly $y(r)$ preserves $R[x]_{\leq c}$, and the map $r \mapsto y(r)$ is an injective homomorphism of $R$ to $\Aut(R[x]_{\leq c})$, so that the semi-direct product 
$$
\Gamma_{c,R} =   R[x]_{\leq c} \rtimes R \cong  R^{c+1} \rtimes R
$$ 
is well defined. 
Once $c$ is fixed, we set   $X_i(R) = \{r x^{c-i} \mid r \in R\} \cong R$ for all  $i\in \{0, 1, \dots, c\}$ and  $Y(R) = \{ y(r) \mid r \in R\} \cong R$, which we view as subgroups of $\Gamma_{c, R}$.

\begin{prop}
\label{prop:generationGamma}
The group $\Gamma_{c,R}$ is nilpotent of class at most $c+1$.  If  $c!$ is a nonzero element in $R$, then the   nilpotency class  
is exactly $c+1$. 
In addition, if $c!$ is invertible in $R$, then the center of $\Gamma_{c, R}$ is $X_c(R)$,  and   $\Gamma_{c,R} = \big\langle X_0(R) \cup Y(R)\big\rangle$.
\end{prop}
\begin{proof}
For $n \in \{0, 1, \dots, c\}$ and $r \in R$, we let $P_n(r) = rx^{c-n} \in R[x]_{\leq c}$ and view it as an element of $\Gamma_{c, R}$. For all $n \in \{0, 1, \dots, c\}$  and $r, s \in R$, we compute the commutator 
\begin{align*}
    [P_{c-n}(r), y(s)] & = 
    -rx^n +  r (x+s)^n\\
    &= -rx^n +  r  \sum_{i=0}^n \binom n i   s^i x^{n-i}\\
    &= r\bigg(ns x^{n-1} + \binom n 2 s^2 x^{n-2} + \dots + n s^{n-1} x + s^n\bigg)\\
    & \in R[x]_{\leq n-1}.
\end{align*}
It follows that $[R[x]_{\leq n}, Y(R)] \leq R[x]_{\leq n-1}$. The groups $R[x]_{\leq c}$ and $Y(R)$ being both abelian, we have  $[\Gamma_{c, R}, \Gamma_{c, R}] = [R[x]_{\leq c}, Y(R)]$, and we infer that $\Gamma_{c, R}$ is indeed nilpotent of class $\leq c+1$. 

The computation above shows that the $k$-fold iterated commutator
$$[\dots\! [[P_0(1), y(1)], y(1)],\dots, y(1)]$$
is a polynomial of degree~$c-k$, and the coefficient of its leading term is equal to $c(c-1)\dots (c-k+1)$. Therefore, if  $c! \neq 0$ in the ring $R$, it follows that the $c+1$-st term of the lower central series of $\Gamma_{c, R}$ contains a non-trivial element, which is actually contained in $X_c(R)$. It follows that the nilpotency class of $\Gamma_{c, R}$ is at least $c+1$, hence it is equal to $c+1$ by the first part of the proof. 

If $c!$ is invertible in $R$, then $n$ is invertible for all $n\in \{1, \dots, c\}$. Given a polynomial $P \in R[x]_{\leq c}$ with degree $d \geq 1$ and leading coefficient $r$, the leading coefficient of $[P, y(1)]$ is $rd \neq 0$. Therefore the center of $\Gamma_{c, R}$ is contained in 
$X_c(R)$, and thus equal to $X_c(R)$ by the first part of the proof above. 

The fact that $\Gamma_{c,R} = \langle X_0(R)  \cup Y(R) \rangle$ when $c!$ is invertible follows by a similar argument. 
\end{proof}

\begin{remark}\label{rem:Comm_Rel_Gamma}
For the sake of future references, we record that, for all $n \in \{0, 1, \dots, c\}$ and all $r, s \in R$, the commutation relation
$$[P_{c-n}(r), y(s)]  
= \sum_{i=1}^n \binom n i    P_{c-n+i}(rs^i)$$
has been established in the course of the proof of Proposition~\ref{prop:generationGamma}.
\end{remark}

\begin{observation}
One can view the groups $\Gamma_{c,R}$ as the $R$ points of a nilpotent group scheme $\Gamma_c$, which is defined over $\Z$, and can be evaluated over all commutative rings. 
In the case $c=1$ the group $\Gamma_{1,R}$ is the Heisenberg group over $R$, which can be viewed as a  maximal unipotent subgroup of $\SL_3(R)$. 
In the case $c=2$ the group $\Gamma_{2,R}$ is isomorphic to a maximal unipotent subgroup of $\Sp_4(R)$. Finally, when $c=3$, the group is
a quotient of the maximal unipotent subgroup of the simple group of type $G_2$ by its center. This  observation will not be needed in the rest of the paper. 
\end{observation}

\subsection{Kazhdan's property (T) and Kazhdan constants}

Let us briefly recall the definition of Kazhdan's property (T). For more details and background, we refer to \cite{BHV}. 

Let $G$ be a discrete group and $(\pi, V)$ be a unitary representation of $G$. Given a subset $Q \subset G$ and a real $\epsilon >0$, a vector $v \in V$ is said to be \textbf{$(Q, \epsilon)$-invariant} if $\sup_{g \in Q} \| \pi(g)v - v\| < \epsilon \| v \|$. We say that $(\pi, V)$ has \textbf{almost invariant vectors} if for all finite subset $Q \subset G$ and all $\epsilon >0$, there is a $(Q, \epsilon)$-invariant vector. We say that $G$ has \textbf{Kazhdan's property (T)} if every unitary representation of $G$ with almost invariant vectors has a non-zero invariant vector. 

Let $Q \subset G$ be a finite subset. Given a unitary representation $\pi$ of $G$, the \textbf{Kazhdan constant associated with $Q$ and $\pi$} is defined as  
$$\kappa(G, Q, \pi) = \inf\, \left\{\max_{g \in Q} \|\pi(g)v - v \| : v \in V, \ \|v\|=1\right\}.$$ 
Thus $\pi$ has almost invariant factors if and only if $\kappa(G, Q, \pi) = 0$ for all finite subsets $Q \subset G$. The infimum of $\kappa(G, Q, \pi) $ taken over all equivalence classes of unitary representations $\pi$ without any non-zero invariant vector is called the \textbf{Kazhdan constant} associated with $Q$. It is denoted by 
$\kappa(G, Q)$. We end this subsection by recording a useful consequence of the definitions. 

\begin{lem}\label{lem:almost-invariant}
Let $G$ be a group and $Q \subset G$ be a generating set $Q$ such that $\kappa(G, Q) >0$. Let $\varepsilon >0$. For any unitary representation $(\pi, V)$ of $G$, each $(Q, \varepsilon)$-invariant unit vector $v \in V$ is also $(G, \frac{2}{\kappa(G, Q)} \varepsilon)$-invariant. 
\end{lem}
\begin{proof}
Let $v = v_0 + v_1$ be the decomposition of $v$ according to the $G$-invariant decomposition $V = V^G \oplus (V^G)^\perp$. If $v_1=0$, then $v$ is $G$-invariant and we are done. We assume henceforth that $v_1 \neq 0$. Since the orthogonal projection map $V \to (V^G)^\perp$ is $1$-Lipschitz, it follows that $v_1$ is $(Q, \varepsilon)$-invariant, hence the unit vector $\frac{v_1}{\| v_1\|}$ is $(Q, \frac{\varepsilon}{\| v_1\|})$-invariant. Since $G$ does not have non-zero invariant vector in $(V^G)^\perp$, we have $\frac{\varepsilon}{\| v_1\|} \geq \kappa(G, Q)>0$, hence $\|v_1\| \leq  \varepsilon/\kappa(G, Q)$. Since $v_0$ is fixed by $G$, we have $\|\pi(g)(v)-v\| = \|\pi(g)(v_1)-v_1\| \leq 2\|v_1\|$ for all $g \in G$. The result follows.  
\end{proof}

\subsection{The representation angle}\label{sec:rep-angle}

We now review a method for proving that a given group has Kazhdan's property (T), and for estimating the Kazhdan constant with respect to a suitable generating set. The genesis of that method goes back to the work by Dymara--Januszkiewicz~\cite{DJ}. It was formalized by Ershov--Jaikin-Zapirain~\cite{EJ}, and improved by Kassabov~\cite{K}. 
The method is based on the notion of the \textbf{angle} formed by subspaces in a Hilbert space. The definitions are as follows.

Let $V$ be a Hilbert space and $V_1,V_2$ be closed subspaces with $V_1 \cap V_2 = \{0\}$. If  $V_1 \neq \{0\} \neq V_2$, the \textbf{angle} formed by $V_1$ and $V_2$, denoted by $\sphericalangle (V_1, V_2)$, is the unique $\alpha \in [0, \pi/2]$ such that $\cos(\alpha) = \sup \{|\langle v_1, v_2\rangle| : v_i \in V_i, \|v_i\|=1\}$.   If $V_1=\{0\}$ or $V_2 = \{0\}$ then we set $\sphericalangle (V_1, V_2) = \frac \pi 2$. 

%\footnote{
Observe that $\cos(\alpha) = \sup \{|\mathrm{Re}(\langle v_1, v_2\rangle)| : v_i \in V_i, \|v_i\|=1\}$. This means that $\sphericalangle (V_1, V_2)$ could equivalently be defined in terms of the geometry of the real Hilbert space $V_{\mathbf R}$, which the real vector space $V$ endowed with the inner product $(v, w) = \mathrm{Re}(\langle v, w\rangle)$.%}

Let now $H$ be a group and  $X, Y \leq H$ be a pair of subgroups such that $H = \langle X \cup Y \rangle$. Given a unitary representation $(\pi, V)$ of $H$ without any non-zero invariant vectors, we denote by $V^X$ and $V^Y$ be the subspaces of $X$- and $Y$-invariant vectors, and we set 
$$
\sphericalangle_\pi (H; X, Y) = \sphericalangle (V^X, V^Y).
$$ 
The infimum of $\sphericalangle_\pi (H; X, Y)$ taken over all equivalence classes of unitary representations $\pi$ without any non-zero invariant vector is called the \textbf{representation angle} associated with $X, Y$. It is denoted by 
$$\sphericalangle (H; X, Y).$$
By the spectral interpretation of the representation angle (see ~\cite{K}), that quantity   behaves well under direct sums: we have $$\sphericalangle_{\oplus \pi_i} (H; X, Y) = \min_i  \sphericalangle_{\pi_i} (H; X, Y).$$
In particular, if $H$ is finite, then the representation angle $\sphericalangle (H; X, Y)$ coincides with the infimum of $\sphericalangle_\pi (H; X, Y)$ taken over all equivalence classes of non-trivial irreducible representations of $H$. 

We also set
$$\varepsilon (H; X, Y) = \cos\big(\sphericalangle (H; X, Y)\big).$$

As mentioned earlier, the machinery initiated in~\cite{DJ} shows that if a group $G$ is generated by a collection $X_1, \dots, X_n$ of finite subgroups  such that 
$$\varepsilon_{ij} =\varepsilon (\langle X_i, X_j \rangle; X_i, X_j) < 2^{-n-1}$$ % \ll 1$$ 
for all $i\not = j$ then the group $G$ has property (T). An explicit, and numerically efficient, incarnation of this phenomenon is provided by 
Theorem 1.2 from \cite{K}, which shows that if the symmetric matrix  
$$
A = \left(
\begin{array}{ccccc}
1 & -\varepsilon_{12} & -\varepsilon_{13} & \dots & -\varepsilon_{1n} \\ 
-\varepsilon_{21} & 1 & -\varepsilon_{23} & \dots & -\varepsilon_{2n} \\
-\varepsilon_{31} & -\varepsilon_{32} & 1 & \dots & -\varepsilon_{3n} \\
\vdots & \vdots & \vdots &  \ddots & \vdots  \\ 
-\varepsilon_{n1} & -\varepsilon_{n2} & -\varepsilon_{n3} & \dots & 1 
\end{array}
\right)
$$
is positive definite, then the group $G$ has property (T). Moreover the Kazhdan constant of $G$
with respect to the generating set  $Q = \bigcup_i X_i$ is related to the smallest eigenvalues of that matrix.

The positive definiteness and the smallest eigenvalue of a symmetric real matrix with non-positive off-diagonal entries can be studied with the tools from \cite[\S4.0--\S4.5]{Kac} that are based on convexity arguments. The set up from loc. cit. also requires the matrix to be indecomposable%
\footnote{The matrix is \textbf{indecomposable} if it can not be conjugated to a block diagonal matrix by a permutation matrix.}, 
which is typically the case for the matrix $A$ above. We illustrate this with the following special case, which is the most relevant for our purposes. 

\begin{prop}\label{prop:pos-def-matrix}
Let $n \geq 2$ be an integer, let $\alpha_1, \dots, \alpha_n>0$ be positive real numbers, and consider the symmetric $n \times n$ matrix
$$
A = \left(
\begin{array}{cccccc}
1 & -\alpha_1 & 0 & \dots & 0 & -\alpha_n \\ 
-\alpha_1 & 1 & -\alpha_2 & \dots & 0 & 0 \\
0 & -\alpha_2 & 1 & \dots & 0 & 0 \\
\vdots & \vdots & \vdots &  \ddots & \vdots & \vdots  \\ 
0 & 0 & 0 & \dots &  1 & -\alpha_{n-1} \\
-\alpha_n & 0 & 0 & \dots & -\alpha_{n-1} & 1 
\end{array}
\right).
$$
Let $\lambda_{\min}$ denote the smallest eigenvalue of $A$ and set
$$
%M = \max\big\{\alpha_i + \alpha_{i+1} \mid i =\{1, \dots, n\}\big\},
M = \max\big\{\alpha_i + \alpha_{i+1} \mid i =1, \dots, n\big\},
$$ 
where the indices are taken modulo~$n$.  The following assertions hold. 
\begin{enumerate}[label=(\roman*)]
    \item  $\lambda_{\min} \geq 1 - M$. In particular, if $M < 1$ then $A$ is positive definite.
    
    \item $\lambda_{\min} = 1 - M$ if and only if $\alpha_i = \alpha_{i+2}$ for all $i$ (modulo $n$). 
    
    \item Assume that $n \geq 4$ and that $\alpha_1 = \dots = \alpha_{n-1} = \alpha$ and $\alpha_n = \beta$. Then $\lambda_{\min} \geq \lambda$ for each $\lambda \in \mathbf R$ satisfying $\lambda \leq 1-2\alpha$, $\lambda < 1-\beta$ and
    $$\alpha^2 \leq (1-\alpha-\lambda)(1-\beta - \lambda).$$
    In particular, if $\alpha < \frac 1 2$, $\beta < 1$ and $\alpha^2 < (1-\alpha)(1-\beta)$ then $A$ is positive definite. 
\end{enumerate}
\end{prop}
\begin{proof}
By hypothesis, the matrix $A$ is symmetric, indecomposable with non-positive off-diagonal entries, so the the conditions from \cite[\S4.0]{Kac} are satisfied. Moreover, the same holds for the matrix $A - \lambda I$ for each $\lambda \in \mathbf R$. Given a vector $v= (v_1, \dots, v_n)^\top \in \mathbf R^n$, we write $v >0$ (resp. $v \geq 0$) if $v_i >0$ for all $i$ (resp. $v_i \geq 0$ for all $i$).  

Assume that there exists $v >0$ in $\mathbf R^n$ with $(A-\lambda I)v \geq 0$. We may then invoke \cite[Theorem~4.3 and Lemma~4.5]{Kac}, showing that two cases can occur:  either $A-\lambda I$ is positive definite, or $A-\lambda I$ is semi-positive definite of rank~$n-1$ and $(A-\lambda I)v = 0$. In both cases we have that $A-\lambda I$ is semi-positive definite, i.e., $\lambda_{\min} \geq  \lambda$. 

Applying this observation to the vector $v= (1, 1, \dots, 1)^\top >0$ and the scalar $\lambda = 1-M$, the assertion (i) follows. 

Moreover, if we assume in addition that $\lambda = 1-M$ is an eigenvalue of $A$, then $A-\lambda I$ is not positive definite, so that we must have $(A-\lambda I)v = 0$ by the above. By the definition of $v$, this implies that $\alpha_i = \alpha_{i+2}$ for all $i$ (modulo $n$). Conversely, if $\alpha_i = \alpha_{i+2}$ for all $i$, then $1-M$ is an eigenvalue of $A$ with eigenvector $v$, so that $\lambda_{\min} = 1-M$ by (i). This proves (ii). 

To prove (iii), we apply the same argument as above, this time with the vector $w= (s, 1, \dots,1, s)^\top$ for some $s > 0$ which remains to be determined. We compute that for any $\lambda \in \mathbf R$, we have $(A-\lambda I)w \geq 0$ if and only if 
$$\left\{
\begin{array}{rcl}
s(1-\beta - \lambda)     &  \geq & \alpha\\
s\alpha      & \leq & 1-\alpha - \lambda\\
\lambda & \leq & 1-2\alpha
\end{array}
\right.$$
%(The third inequality occurs only if $n=4$.) 
We see that a real number $s >0$ satisfying those inequalities exists provided $\lambda$ satisfies $\lambda \leq 1-2\alpha$, $\lambda < 1-\beta$ and $\alpha^2 \leq (1-\alpha-\lambda)(1-\beta - \lambda)$. If these three conditions hold, then we obtain $\lambda_{\min} \geq  \lambda$ by the first paragraph above. Clearly, if $\alpha < \frac 1 2$, $\beta < 1$ and $\alpha^2 < (1-\alpha)(1-\beta)$, then any sufficiently small 
$\lambda >0 $ satisfies those three inequalities, so that $A$ is indeed positive definite in that case.
\end{proof}

\subsection{Finite fields}

We need a few technical results about finite fields, which have relatively easy proofs. Throughout this section, we let $p$  denote a  prime. Given a field $k$ containing $\F_p$ and an element $\beta \in k$, we denote by $\F_p(\beta)$  the subfield of $k$   generated by the element $\beta$.

\begin{lem}
\label{lem:count}
Let $N\geq 1$ be an integer, let  $p> N$ be a prime and let $\F_q$ denote a finite field of order $q=p^l$. For each non-zero $\gamma \in \F_q$,   the proportion of those elements $\alpha \in \F_q$ such that 
$\F_q = \F_p(\alpha^N\gamma)$, is at least $1 - N /p$.
\end{lem}
\begin{proof}
Without loss of generality we can assume that $l\geq 2$.
For each integer $l'$ dividing  $l$, the number of elements $\alpha \in \F_q$ such that $\F_p(\alpha^N\gamma) \subset \F_{p^{l'}}$ is at most $N p^{l'}$.
Therefore the number of $\alpha \in \F_q$ such that $\alpha^N \gamma$ does not generate $\F_q$ is at most
$$
\sum_{l' | l, \, l'<l} N p^{l'} \leq \sum_{l'=1}^{l/2 } N p^{l'} = N \frac{p^{\frac l 2+1} -p}{p-1} \leq \frac{N}{p} p^{l}.
$$
\end{proof}

\begin{lem}
\label{claim:enlarge}
Let $\alpha,\beta \in \overline{\F_p}$ with $\alpha \neq 0$,  
and 
$k, N$ be  integers with  $0\leq  k < N $.
Assume that $p > N $.

\begin{enumerate}[label=(\roman*)]
    \item There exists $\lambda \in \F_p(\alpha^N)$ such that  $|\F_p\big((\beta + \lambda \alpha^k)^N\big)| \geq | \F_p(\alpha^N)|$.
    
    \item If $\F_p(\alpha^N) \neq \F_p(\alpha^N,\beta^N,\alpha^{k}\beta^{N-1})$,
then there exists $\lambda \in \F_p(\alpha^N)$ such that $|\F_p\big((\beta + \lambda \alpha^k)^N\big)| > | \F_p(\alpha^N)|$.

\end{enumerate}
\end{lem}
\begin{proof}
The proof consists of a counting argument. 
The strategy is the following:  if the the size of the field $\F_p\big((\beta + \lambda \alpha^k)^N\big)$ is sufficiently small, then $\lambda$ is a root of some equation of small degree. 
The sum of the degrees of all these equations is smaller than the size of $\F_p(\alpha^N)$, which shows that there exists an element $\lambda$ which is not a root of any of these equations.

Let $q =p^l =  |\F_p(\alpha^N)| $. We shall count the number of elements $\lambda \in \F_p(\alpha^N)$ such that the field generated by $\beta'^N$ has $p^t \leq q $ elements for some $t\leq l$, where $ \beta' = \beta + \lambda \alpha^k$.

If $|\F_p(\beta'^N)| = p^t$, it follows that $\beta'^{Np^t} = \beta'^N$. Substituting $\beta' = \beta + \lambda \alpha^k$ in the latter equality, we obtain  an equation  of degree $Np^t$ for $\lambda$, which has no more than $Np^t$ roots in $\F_p(\alpha^N)$. For $t=1, \dots, l-1$, we obtain at most $\sum_{t=1}^{l-1} N p^t = N \frac{p^l-p}{p-1}$ possible values for $\lambda$. Since $N \leq p-1$, we infer that there are at least $p$ values of $\lambda \in \F_p(\alpha^N)$ such that $|\F_p\big((\beta + \lambda \alpha^k)^N\big)| \geq | \F_p(\alpha^N)|$. This proves (i). 

In order to prove (ii), we need to evaluate the number of those $\lambda$'s such that  $|\F_p\big((\beta + \lambda \alpha^k)^N\big)| = | \F_p(\alpha^N)|$.  Hence, we now assume that $|\mathbf F_p(\beta'^N)| = |\mathbf F_p(\alpha^N)|$, so that $\beta'^{Nq} = \beta'^N$. We shall slightly modify the previous argument using that $\lambda^{q} = \lambda$, in order to obtain an equation of degree $N-1$ for $\lambda$ as follows. Observe that 
\begin{align*}
\beta'^{Np^l} & = \left(\beta'^{p^l}\right)^N \\
& = \left( \beta^{p^l} + \lambda^{p^l} \alpha^{kp^l}\right)^N \\
& =    \left( \beta^{p^l} + \lambda \alpha^{kp^l}\right)^N \\
& = \beta^{Np^l} + \dots +  \lambda^N \alpha^{kNp^l}
\end{align*}
and that
\begin{align*}
\beta'^N 
& = \left( \beta + \lambda \alpha^k\right)^N \\
&= \beta^{N} + \dots +  \lambda^N \alpha^{kN}.
\end{align*}
Recalling that $\beta'^{Np^l} = \beta'^N$, we obtain an equation of degree $N$ for $\lambda$. Moreover, 
since $\alpha^N \in \F_{p^l}$, the coefficients in front of $\lambda^N$ in both sides of the equation  cancel out.
Observe that the resulting equation of degree $N-1$ for $\lambda$ is non-trivial. Indeed, if the independent term and the coefficient for the term of degree~$1$ both vanished, we would deduce that $\beta^{Nq} = \beta^N$ and that $(\alpha^k\beta^{N-1})^q =\alpha^k\beta^{N-1}$, so that $\beta^N$ and $\alpha^k\beta^{N-1}$ would be both contained in  $\F_p(\alpha^N)$. This would contradict the hypothesis that $\F_p(\alpha^N)$  is strictly contained in  $\F_p(\alpha^N,\beta^N,\alpha^{k}\beta^{N-1})$.

Overall, it follows that  the number of $\lambda$'s such that $\F_p(\beta'^N)$ has size at most $q$ is bounded above by
$$
N-1  + \sum_{t=1}^{l-1} N p^t < N + N\left(\frac{p^l - p}{p-1}\right) \leq  p-1 +  p^l-p = p^l - 1 = q-1
$$
since $N \leq p-1$ by hypothesis. 
Therefore, there exists $\lambda \in \F_p(\alpha^N)$ such that $\F_p(\beta'^N)$ has more than $q$ elements, which proves (ii).
\end{proof}

\subsection{Polynomial functions with prescribed values}

The following elementary result may be viewed as a consequence of the Chinese remainder theorem. 

Since it plays a key role in the sequel, we include a (short) direct proof for the reader's convenience. 

\begin{lem}
\label{lem:polynomial-map-k-tuple}
Let $K$ be an arbitrary field and $\overline{K}$ denote an algebraic closure. 
Let $\mu_1, \dots, \mu_k \in \overline{K}$ be elements whose respective minimal polynomials over $K$ are pairwise distinct. 
Then for all $\nu_1, \dots, \nu_k \in \overline{K}$ such that $\nu_i \in K(\mu_i)$, there exists a polynomial $f \in K[x]$ satisfying 
$f(\mu_i) = \nu_i$ for all $i \in \{1, \dots, k\}$. 
\end{lem}
\begin{proof}
We work by induction on $k$. 
In the base case where $k=1$, we observe that the elements $1, \mu_1, \mu_1^2, \dots, \mu_1^{d-1}$ form a basis of $K(\mu_1)$ viewed as a vector space over $K$, where $d = [K(\mu_1):K]$. The required assertion follows. 

Let now $k \geq 2$ and assume that the result is true for $k-1$. For each $i \in \{1, \dots, k\}$, let $f_i \in K[x]$ be the minimal polynomial of $\mu_i$. The hypothesis of the lemma implies that $f_i(\mu_j) \neq 0$ for all $i \neq j$. By the induction hypothesis, there exists a polynomial $\varphi \in K[x]$ such that 
$$\varphi(\mu_i) = \frac{\nu_i}{f_k(\mu_i)} \hspace{1cm} \text{for all } i = 1, \dots, k-1.$$
The induction hypothesis also ensures the existence of a polynomial $\psi \in K[x]$ such that 
$$\psi(\mu_k) = \frac{\nu_k}{f_1(\mu_k)f_2(\mu_k)\dots f_{k-1}(\mu_{k})}.$$
The polynomial $f(x) = f_k(x)\varphi(x) + f_1(x)f_2(x)\dots f_{k-1}(x)\psi(x)$ belongs to $K[x]$ and satisfies the required property.
\end{proof}

\section{Bounding  representation angles for  some nilpotent groups}

Various results providing bounds on the representation angle  for nilpotent groups have been established, starting with \cite[\S 4]{EJ} which deals with groups of nilpotency class~$2$. Further results may be found in \cite[\S 10]{EJK}. In general,  the bounds on the angle depend on the nilpotency class of the group and on the size of the smallest finite quotient. By way of illustration, we record the following proposition that can be deduced from~\cite[Theorem~4.1]{ER18}.   

\begin{prop}\label{prop:rep-angle}
Let $c \geq 1$ be an integer. There is a constant $C$ (depending on $c$) such that for any prime $p$ and any nilpotent group $G$  of class $c$ generated by a pair $X, Y$ of subgroups of order equal to a power of $p$, 
we have  $$\varepsilon(G; X, Y) \leq C/\sqrt[2^{c-1}]{p}.$$ 
\end{prop}
\begin{proof}
We work by induction on $c$. For $c=1$ the group is abelian and we have $\varepsilon(G; X, Y) = 0$ since all irreducible unitary representations of $G$ are $1$-dimensional. 

Let now $G$ have nilpotency class $c+1$ and let $H$ be the center of $G$. The hypotheses imply that $G$ is a finite $p$-group. Therefore, the minimal degree of  an irreducible representation of $G$  whose restriction to $H$ is non-trivial is at least $p$. By \cite[Theorem~4.1]{ER18}, we have $\varepsilon(G; X, Y) \leq \sqrt{\varepsilon(G/H; X', Y')+\frac 1 p}$, where $X' = XH/H$ and $Y' = YH/H$. By induction, there is a constant $C$ such that $\varepsilon(G/H; X', Y') \leq C/\sqrt[2^{c-1}]{p}$. Using that $ p \geq  \sqrt[2^{c-1}]{p}$, we infer that 
$\varepsilon(G; X, Y) \leq \sqrt{C+1}/\sqrt[2^c]{p}$, as required.
\end{proof}

%\begin{remark}\label{rem:nilpotent}
%Observe that the proof provides an explicit estimate of the constant $C$, for example $C=1$ for $c=2$ and $C=\sqrt 2$ for $c=3$. 
%\end{remark}

This can be applied to the group $G = \Gamma_{c,\F_q}$, with $X = X_0(\F_q)$ and $Y = Y(\F_q)$ for any $q= p^e$. 
In case of the group $\Gamma_{c,R}$, this bound can be improved using relatively elementary argument provided the ring 
$R$ is countable. 
Notice that the bound below is stronger than the one in Proposition~\ref{prop:rep-angle} since it decays as $p^{-1/2}$ which is much smaller than than $p^{-1/2^{c-1}}$ as $p$ tends to $\infty$.

\begin{thm}
\label{thm:rep-angle-improved}
Let $R$ be a finite or countable commutative ring with a unit. Denote by $p$ be the smallest prime $p$ such that $p$ is not invertible in $R$ (we set $p = \infty$ if $R$ contains the rational numbers $\Q$). For any integer $c$ with $1 \leq c \leq p-1$,  we have 
$$
\varepsilon( \Gamma_{c,R}; X_0(R) , Y(R) ) \leq  \sqrt{c/p}.
$$
\end{thm}

For the proof, we need the following elementary fact.

\begin{lem}\label{lem:aux2}
Let $(Z, \nu)$ be a probability space and $c, p$ be integers with $1\leq c < p$. Let $A_1, \dots, A_p \subseteq Z$ be subsets of equal measure and assume that each $z \in Z$ belongs to at most $c$ of the sets $A_1, \dots, A_p$. Then $\nu(A_i) \leq c/p$. 
\end{lem}
\begin{proof}
For each integer $n \geq 0$, let $Z(n) \subseteq Z$ be the subset consisting of those $z \in Z$ which belong to exactly $n$ of the sets $A_1, \dots, A_p$. By hypothesis, we have $Z = Z(c) \cup Z(c-1) \cup \dots \cup Z(0)$. Moreover, the sets $A_i$ have equal measure, so that
\begin{align*}
p \nu(A_i) & = \sum_{k=1}^p \nu(A_k)\\
& = c \nu(Z(c)) + (c-1) \nu(Z(c-1)) + \dots + \nu(Z(1))\\
& \leq c\nu\big(Z(c) \cup \dots \cup Z(1)\big)\\
& \leq c \nu(Z)\\
&= c.
\end{align*}
The desired assertion follows. 
\end{proof}

\begin{proof}[Proof of Theorem~\ref{thm:rep-angle-improved}]
Let $(\pi, V)$ be a unitary representation of $\Gamma_{c,R}$ without any non-zero invariant vectors.

Since $X$ is normal in $\Gamma = \Gamma_{c, R}$, we have $V = V^X \oplus (V^X)^\perp$ as $\Gamma$-representations. Moreover,  as mentioned above,   the representation angle behaves well with respect to direct sums, i.e. we have $\sphericalangle_{\oplus \pi_i} (H; X, Y) = \min_i  \sphericalangle_{\pi_i} (H; X, Y)$.  
Therefore, we only need to consider separately the  cases where $V$ has no non-zero $X$-invariant vectors, and where the $X$-action on $V$ is trivial.
In the second case $V$ does not have any non-zero $Y(R)$-invariant vectors, thus the representation angle is $\pi/2$ by definition. We assume henceforth that $V$ does not contain any non-zero $X$-invariant vectors.

We invoke the SNAG Theorem for the restriction of the representation to the abelian group  $X$, see \cite[Theorem~2.C.3]{BH_duals}, as well as \cite[Appendix~A.9]{BH_duals}.  
This yields a projection-valued measure 
$\mu$ defined on the dual $\widehat X = \Hom(X, \mathbf S^1)$, with values in the set of orthogonal projections on $V$.
For any non-zero vector $v\in V$, the assignment  
$$\mu_v(E) = \langle v, \mu(E) v \rangle$$
defines a %Radon 
measure $\mu_v$ on $\widehat X$. Moreover, if $v$ is a unit vector, then  $\mu_v$ is a probability measure, see  \cite[Appendix~A.9]{BH_duals}.

Since $V$ has  no non-zero $X$-invariant vectors, it follows that the trivial character $\mathbf 1_X \in \widehat X$ is not an atom of the measure $\mu$, see \cite[Prop.~2.D.1(a)]{BH_duals}.
Let 
$$
X_0^\perp  = \{\chi \in \widehat X \mid \chi(g) = 1 \ \forall g \in X_0\}\leq \widehat X
$$ 
be the annihilator of $X_0$. By applying \cite[Prop.~2.D.1(a)]{BH_duals} to the restriction of the representation to $X_0$, we see that the operator $\mu(X_0^\perp)$ coincides with the orthogonal projection on the subspace $V^{X_0}$ consisting of the $X_0$-invariant vectors. In particular, it follows that for each unit vector $v \in V$, the cosine of the angle between $v$ and $V^{X_0}$ is equal to $\sqrt{\mu_v(X_0^\perp)}$.

The group $X$ is normal in $G$ therefore $G$ acts on the dual $\widehat X$. 
In particular we have an action of the integers $\Z$ on $\widehat X$ via the natural homomorphism $\Z \to Y(R): t \mapsto Y(t.1)$ and the action of $
Y(R)$ on $X$
given by conjugation.

\begin{lem}\label{lem:aux1}
Let $\psi \in \widehat X$ be a nontrivial character. Then the number of integers $t \in \{0,1,\dots,  p-1\}$  such that $Y(t.1) \psi \in X_0^\perp$ is at most $c$.
\end{lem}
\begin{proof}
The abelian group $X$ decomposes as a direct sum  $X = \bigoplus_{i=0}^c X_i(R)$. Hence it  can be identified with the direct sum of $c+1$ copies of the additive group $R$, and the dual $\widehat X$ can be identified with the direct sum of $c+1$ copies of $\widehat R$.  Under this identification, we write $\psi = (\psi_0,\psi_1, \dots, \psi_c) \in \bigoplus_{i=0}^c \widehat R$, and the action of $\Z$ is given by
$$
Y(t.1) ( \psi_0,\psi_1, \dots, \psi_c) =  ( \psi'_0,\psi'_1, \dots, \psi'_c),
$$
where 
$\psi'_n = \sum_{i=0}^{c-n} \binom{c-n} i t^i \psi_{n+i}$.

Assume that there exist pairwise distinct integers
 $t_0, t_1, \dots, t_c \in  \{0,1,\dots,  p-1\}$ such that $Y(t_s.1) \psi  \in X_0^\perp$ for all $s$. We must deduce that $\psi$ is the trivial character. 
 
For each $s \in \{0, 1, \dots, c\}$, let $\phi_s$ denote the first coordinate of
$Y(t_s.1) \psi = Y(t_s.1)( \psi_0,\psi_1, \dots, \psi_c)$, so that $\phi_s = \sum_{i=0}^{c} \binom{c} i (t_s)^i \psi_{i}$. 
Considering those  equations for $s=0, 1, \dots, t$, we obtain the equation 
$(\phi_0, \dots, \phi_t)^\top = 
%V \big(\binom c 0 \psi_0, \dots, \binom c c \psi_c \big)^\top$, 
V \big(\dots, \binom c i \psi_i, \dots \big)^\top$, 
where $V$ is a Vandermonde matrix. By hypothesis, we have $p>c$ and the number $c!$ is invertible in $R$, so that   the binomial coefficients $\binom c i$ are invertible in $R$ for all $i$. Moreover 
$t_s - t_{s'}$  is also invertible for all $s\not =s'$. Therefore, we may express all the $\psi_i$ as linear combinations of the $\phi_s$.

By the definition of $\phi_s$, the condition that $Y(t_s.1) \psi$ belongs to the annihilator $X_0^\perp$ is equivalent to saying that $\phi_s =0$ in $\widehat R$. 
The preceding discussion then implies that $\psi_i=0$ for all $i$. In other words $\psi$ is the trivial character.
\end{proof}

The above lemma allows us to complete the proof of Theorem~\ref{thm:rep-angle-improved} as follows. First observe that 
$$\varepsilon(\Gamma_{c, R};X_0, Y(R)) \leq \cos\big(\sphericalangle(V^{X_0}, V^{Y(1)})\big)$$
since $V^{Y(R)} \subseteq V^{Y(1)}$. 
Therefore, it suffices to provide an upper bound on the cosine of  the angle formed by  $V^{X_0}$ and $V^{Y(1)}$. 
Consider an arbitrary unit vector $v \in V^{Y(1)}$. Since $v$ is fixed by $Y(1)$, the measure $\mu_v$ is  invariant under  $Y(t.1) $ for all $t \in \mathbf Z$.
Therefore $\mu_v(Y(t.1) X_0^\perp) = \mu_v(X_0^\perp)$ for any $t\in \mathbf Z$. The previous lemma gives that any non-trivial character in $\widehat X$ belongs to at most $c$  of the sets $\mu_v(Y(t.1) X_0^\perp)$, where $t=0,\dots, p-1$. Since $\mathbf 1_X \in \widehat X$ is not an atom of the measure $\mu_v$, we may invoke Lemma~\ref{lem:aux2} for the probability space $(\widehat X \setminus \{\mathbf 1_X\}, \mu_v)$. We deduce   
that  $\mu_v(X_0^\perp)\leq c/p$. This finishes the proof since we have seen above that 
$$\cos\big(\sphericalangle(V^{X_0}, v) \big)= \sqrt{\mu_v(X_0^\perp)},$$ so that $\sqrt{c/p}$ is indeed an upper bound for the cosine of the   angle $\sphericalangle(V^{X_0}, V^{Y(1)})$.
\end{proof}

\section{On tame automorphisms of the affine space}

Let $R$ be a commutative 
ring with a unit and consider the polynomial ring 
$$R_n=R[x_1,\dots, x_n].$$
For any $1 \leq i \not = j \leq n$ and any integer $e \geq 0$, we define  $\trans ijer \colon  R_n\to R_n$
as  the automorphism of
$R_n$ such that
\begin{align*}
 \trans ijer: \left\{
 \begin{array}{ll}
    x_i \mapsto x_i + r x_j^e   &  \\
    x_\ell \mapsto x_\ell & \mbox{for } \ell \not =i.  
 \end{array}\right.
\end{align*}
%Those automorphisms can be thought of as \texit{polynomial transvections}. 
For a fixed triple $i,j,e$, the subgroup 
$$
\Trans ije = \left\langle \trans ijer \mid r \in R\right\rangle
$$ 
is isomorphic to the additive group of the ring $R$. 

We recall that, throughout this paper, we use the notation $[g, h] = g^{-1} h^{-1} gh$ for the commutator of $g$ and $h$. 
The automorphisms $\trans ijer$ satisfy the following commutation relations that are easy to verify, where different symbols are assumed to represent  different indices:
\begin{itemize}
    \item $[\trans ijer, \trans {i'}{j'}{e'}{r'}] =\id$,
    \item $[\trans ijer, \trans {i}{j'}{e'}{r'}] =\id$,
    \item $[\trans ijer, \trans {i'}{j}{e'}{r'} ] =\id$.
\end{itemize}

We shall next collect important information on the subgroup of $\Aut(R_n)$ generated by $\Trans ije$ and $\Trans jk{e'}$. To that end, 
 it is useful to define a few more elements in $\Aut(R_n)$. Let 
$\dtrans ijkcdr \colon  R_n\to R_n$ be  the automorphism of
$R_n$ such that
\begin{align*}
 \dtrans ijkcdr: \left\{
 \begin{array}{ll}
    x_i \mapsto x_i + r x_j^cx_k^d   &  \\
    x_\ell \mapsto x_\ell & \mbox{for } \ell \not =i.  
 \end{array}\right.
\end{align*}
For a fixed tuple $i,j,k,c,d$, the subgroup $\dTrans ijkcd$ generated by $\dtrans ijkcdr$ is isomorphic to the additive subgroup of the ring $R$. 
It is easy to see that these elements satisfy commutation relations similar to the ones satisfied by $\trans ijer$.

In the following result, we retain the notation of Proposition~\ref{prop:generationGamma} concerning the group $\Gamma_{c, R}$. In particular $P_\ell(r)$ denotes the monomial $rx^{c-\ell}$, viewed as an element of $\Gamma_{c, R}$. 

\begin{prop}
\label{prop:nilpotent}
Let  $c \geq 1$ and $d \geq 0$ be integers. Let $i, j, k \in \{1, \dots, n\}$ be pairwise distinct.  Then the assignments 
$$P_\ell(r) \mapsto \dtrans ijk{c-\ell}{d\ell}r 
\qquad \text{and} \qquad 
y(s) \mapsto \trans jkds,$$
where $r,s \in R$ and $\ell \in \{0, 1, \dots, c\}$, 
extend to an isomorphism
$$\Gamma_{c, R} \to \left\langle (\bigcup_{\ell=0}^c \dTrans ijk{c-\ell}{d\ell}) \cup \Trans jkd\right\rangle \leq \Aut(R_n).$$
In particular, if $c!$ is an invertible element in $R$, then $\left\langle \Trans ijc  \cup \Trans jkd \right\rangle$ is isomorphic to $\Gamma_{c, R}$ and contains $\dtrans ijk{c-\ell}{d\ell}r$ for all $r \in R$ and $\ell \in \{0, 1, \dots, c\}$.
\end{prop}
\begin{proof}
Let us first observe that the given assignments establish an isomorphism between the abelian subgroup $R[x]_{\leq c} \leq \Gamma_{c, R}$ and $\left\langle \bigcup_{\ell=0}^c \dTrans ijk{c-\ell}{d\ell}\right\rangle$. We must now compare the respective conjugation actions of $y(s)$ and $\trans jkds$ on those abelian groups. To this end, we consider non-negative integers  $m, n, d$  and elements  $r, s \in R$. We compute that 
\begin{align*}
[\dtrans ijk{m}{n}r, \trans jkds] : \left\{
\begin{array}{ll}
  x_i \mapsto x_i + r x_k^n \left((x_j+sx_k^d)^m - x_j^m\right)   &  \\
    x_\ell \mapsto x_\ell & \mbox{for } \ell \not =i.  
\end{array}
\right.
\end{align*}
Using an additive notation for the subgroup of $\Aut(R_n)$ generated by $\dtrans ijk{m'}{n'}r$ with $i, j, k$ fixed, we infer that 
$$
[\dtrans ijk{m}{n}r, \trans jkds] = \sum_{t=1}^m \binom m t \dtrans ijk{m-t}{n+dt}{rs^t}.
$$
It follows from Remark~\ref{rem:Comm_Rel_Gamma} that the assignments 
$P_\ell(r) \mapsto \dtrans ijk{c-\ell}{d\ell}r$  and $y(s) \mapsto \trans jkds$ indeed extend to an isomorphism as required.  If in addition $c!$ is invertible in $R$, we deduce the extra conclusions follow from Proposition~\ref{prop:generationGamma}. 
\end{proof}

\begin{coro}
\label{cor:nilpotent}
Let $p$ be a prime such that $p>c$ and set $R= \F_p$. 
Let $i, j,  k \in \{1, \dots, n\}$  be pairwise distinct, and let $c \geq 1$ and $d \geq 1$ be integers. The assignments
$$x \mapsto \trans ijc1 \qquad \text{and} \qquad y \mapsto \trans jkd1$$
extend to an injective group homomorphism 
$$\Gamma_{c,\F_p} \to \Aut(\F_p[x_1, \dots, x_n]). $$
\end{coro}

The computation in the proof of  Proposition~\ref{prop:nilpotent} actually establishes the following technical variant of that statement.

\begin{prop}
\label{prop:nilpotent:variant}
If $c!$ is an invertible element in $R$, then  
$\left\langle \dTrans ijkcd \cup \Trans jke\right\rangle$ is isomorphic to $\Gamma_{c,R}$.
That subgroup of $\Aut(R_n)$ contains the automorphisms $\dtrans ijk{c-\ell}{d+e\ell}r$ for all $r \in R$ and $\ell \in \{0, 1, \dots, c\}$.
\end{prop}
\begin{proof}
The same computation as in the proof of Proposition~\ref{prop:nilpotent} shows that the assignments $P_\ell(r) \mapsto \dtrans ijk{c-\ell}{d+e\ell}r$  and $y(s) \mapsto \trans jkds$  extend to the required isomorphism.
\end{proof}

\section{A family of Kazhdan groups of tame automorphisms}\label{sec:Kazhdan}

Let $n \geq 3$, let $e_1, \dots, e_n \geq 1$ be positive integers and $R$ be a commutative ring with a unit. We set $\mathbf e = (e_1, \dots, e_n)$.
Our main object of study is the subgroup $G_{R,\mathbf{e}}$ of $ \Aut(R_n)$ defined by 
$$
G_{R,\mathbf{e}} = \left\langle 
\trans i{i+1}{e_i}r \, \big| \,i=1, 2, \dots, n;\  r\in R 
\right\rangle, 
$$
where indices are taken modulo~$n$. The number $n$ is called the \textbf{rank} of the group $G_{R,\mathbf{e}}$. To lighten the notation, we set 
$\tau_i(r) = \trans i{i+1}{e_i}r$ and $X_i = \langle \tau_i(r) \rangle = \Trans i{i+1}{e_i}$ 
for all $i \mod n$. 

Our next goal is to establish the following.

\begin{thm}\label{thm:T-for-G}
Let $n \geq 3$ and $e_1, \dots, e_n \geq 1$ be positive integers. 
Let $p$ be an integer such that $(p-1)!$ is invertible in $R$, $p> e_i$ for all $i$ and that 
$$
M = \max\left\{\sqrt{e_i/p} + \sqrt{e_{i+1}/p}\, \big|\, i=1, \dots, n\right\}<1,
$$
where indices are taken modulo~$n$. For each $i\in \{1, \dots, n\}$, let $X_i =\Trans i{i+1}{e_i}$.%\{\tau_i(r) \, | \, r\in R \}$.
Then the Kazhdan constant of the  group $G_{R,\mathbf{e}}$ with respect to the generating set 
$S = \bigcup_i X_i$ is bounded below by
$$
\kappa(G_{R,\mathbf{e}}, S ) \geq 
\sqrt{\frac{2(1-M)}{n}} >0.
$$
If in addition $|R| = q < \infty$, then 
$$
\kappa(G_{R,\mathbf{e}}, S ) \geq 
\sqrt{\frac{2q}{q-1}} \sqrt{\frac{1-M}{n}}.
$$
In particular, if $R$ is  finite 
and if $p > 4 \max \{ e_i\}$, 
then $G_{R,\mathbf{e}}$  has  property (T).
\end{thm}
\begin{proof}
Each $X_i$ is a subgroup of $G=G_{R,\mathbf{e}}$ isomorphic to the additive subgroup of $R$.
For any $i, j$ with $|i-j| \geq 2$ the subgroups $X_i$ and $X_j$ commute, thus the representation angle between them is $0$.
The subgroup $\langle X_i, X_{i+1}\rangle$ is isomorphic to $\Gamma_{e_i,R}$ by Proposition~\ref{prop:nilpotent}, thus the cosine of the  representation angle is 
bounded above by $\sqrt{e_i/p}$ by Theorem~\ref{thm:rep-angle-improved}.
This allows us to invoke Kassabov's results from \cite{K} recalled in Section~\ref{sec:rep-angle}. 
%More precisely, we use \cite[Theorem~5.1]{K} in a similar way as in the proof of \cite[Theorem~6.1]{K}.  
To this end, consider a unitary representation $(\pi, V)$ of $G$ that does not contain any non-zero invariant vector. Let $\varepsilon >0$ and $v \in V$ be a unit vector that is $(S, \varepsilon)$-invariant. Let $d_i$ denote the distance from $v$ to the subspace $V^{X_i}$ of $X_i$-invariant vectors. Since $(\pi, V)$  does not contain any non-zero $G$-invariant vector, we deduce from \cite[Theorem~5.1]{K} that 
$$1 \leq (d_1, \dots, d_n) A^{-1} (d_1, \dots, d_n)^\top,$$ 
where $A$ is the real matrix defined by 
$$A = 
\left(
\begin{array}{cccccc}
1 & -\sqrt{e_1/p} & 0 & \dots & 0 & -\sqrt{e_n/p} \\ 
-\sqrt{e_1/p} & 1 & -\sqrt{e_2/p} & \dots & 0 & 0 \\
0 & -\sqrt{e_2/p} & 1 & \dots & 0 & 0 \\
\vdots & \vdots & \vdots &  \ddots & \vdots & \vdots \\ 
0 & 0 & 0 &\dots & 1 & -\sqrt{e_{n-1}/p}\\
-\sqrt{e_n/p} & 0 & 0 & \dots & -\sqrt{e_{n-1}/p} & 1 
\end{array}
\right).
$$
 In view of Proposition~\ref{prop:pos-def-matrix}(i), we know that the smallest eigenvalue of $A$, say $\lambda_{\min}$, satisfies $\lambda_{\min} \geq 1-M$. By the definition of $v$, any two vectors in the $X_i$-orbit of $v$ are at distance at most~$\varepsilon$ apart, so that $d_i \leq \varepsilon/ \sqrt{2} $ for all $i$. We infer that 
 \begin{align*}
     1 &\leq (d_1, \dots, d_n) A^{-1} (d_1, \dots, d_n)^\top\\
     & \leq \frac{1}{2} \varepsilon^2 \|(1, \dots, 1)\|^2 \lambda_{\min}^{-1}\\
     & \leq \frac{1}{2} \varepsilon^2 n (1-M)^{-1}.
 \end{align*}
 This directly implies that the Kazhdan constant satisfies 
 $\kappa(G_{R,\mathbf{e}}, S ) \geq \varepsilon \geq 
\sqrt{\frac{2(1-M)}{n}}$.

Assume now that $R$ is finite of order $q$. Then $X_i$ is a group of order $q$ for all $i$, and the $X_i$-orbit of $v$ is contained in an affine subspace of $V$ of dimension~at most $q-1$. By Jung's theorem \cite{Jung}, in the $(q-1)$-dimensional Euclidean space, the circumradius $R$ of any subset of diameter $D$ satisfies the inequality $R \leq D \sqrt{\frac{q-1}{2q}}$. Therefore, in the present case, the distance $d_i$ satisfies $d_i \leq \varepsilon \sqrt{\frac{q-1}{2q}}$. Now the inequality
$\kappa(G_{R,\mathbf{e}}, S ) \geq 
\sqrt{\frac{2q}{q-1}} \sqrt{\frac{1-M}{n}}$ follows from the same argument as above.
\end{proof}

The previous result does not give property (T) when the ring $R$ is infinite, since the generating set $S$ is not finite. 
However, it is known~\cite{EJ,Shalom} that the subgroup $\EL_n(R) \leq \SL_n(R)$ generated by elementary matrices has Kazhdan's property $(T)$ provided $R$ is finitely generated and  $n\geq3$. Observe that $\EL_n(R)$ may be viewed as the subgroup of $\Aut(R[x_1,\dots, x_n])$ generated by the set $\{\alpha_{i; j}^{(1)}(r) \mid 1\leq i \neq j \leq n, \ r \in R\}$. 
Relying on the fact that $\EL_n(R)$ has property (T),  we establish  the following.  

\begin{thm}
\label{thm:T-bis}
For any $n\geq 3$ and any finitely generated ring $R$, such that $30=2\times 3 \times 5$ is invertible in $R$, the subgroup $G_{n,R}$ of
$\Aut(R[x_1,\dots, x_n])$ generated by $\EL_n(R)$ and $\trans 1221$ has Kazhdan property (T).
If $n \geq 4$, the same conclusion holds as soon as $6$ is invertible in $R$. 
\end{thm}

\begin{proof}
Let $p$ denote the smallest prime which is not invertible in the ring $R$.
By hypothesis, we have  $p \geq 7$.  
Let $S'$ be a finite generating set of $G_{n,R}$, which contains a generating set $S''$ of $\EL_n(R)$ and $\trans 1221$, and let
$\rho$ be a representation of $G_{n,R}$ on a Hilbert space $\mathcal{H}$ with an $\varepsilon$-almost invariant unit vector $v$ under $S$, for some sufficiently small $\varepsilon$. 
Using Kazhdan's property (T) for $\EL_n(R)$, we  deduce from Lemma~\ref{lem:almost-invariant} that the vector $v$ is $C_n\varepsilon$-almost invariant under the whole group $\EL_n(R)$ for some constant $C_n \geq 1$. In particular $v$
is almost fixed by the elementary matrices $\trans ij1r$.

We claim that there exists constant $D_n \geq   C_n$ such that $v$ is $D_n\varepsilon$ almost invariant under the subgroups  
$\Trans ije$ for $1 \leq i,j \leq n$ for all $i\not=j \in \{1, \dots, n\}$ and all $e \in \{1,2\}$. Indeed,  for $e=1$, we have $\trans ij1r \in \EL_n(R)$, so that $\Trans ij1$ almost fixes the vector $v$. For $e=2$, we first remark that $v$ is $\varepsilon$-almost invariant under $\trans 1221$  by the assumption that $S'$ contains  that element. Using that
$\trans 322r = \left[\trans 311r, \trans 1221\right] $, we see that  all elements in $\Trans 322$ can be expressed by products of at most~$4$ group elements which almost fix $v$. This implies that $v$ is $(2C_n+2)\varepsilon$-almost invariant under $\Trans 322$. Finally, for any pair $i\neq j \in \{1, \dots, n\}$, the group $\Trans ij2$ is conjugate to $\Trans 322$ by a suitable permutation matrix, which belongs to  $\EL_n(R)$. Again, this implies that all elements in $\Trans ij2$ can be expressed by products of three group elements which almost fix $v$. The claim follows, with 
$D_n= 4C_n +2$.

We next observe that the group $G_{n,R}$ is generated by  the abelian subgroups $Y_1= \Trans 121$, $Y_2=\Trans 231$,\dots, $Y_{n-1}= \Trans {n-1}n1$ and $Y_n = \Trans n11 \Trans n12 \cong R^2$.
Let us now estimate the cosine of the representation angle $\varepsilon(\langle Y_i,Y_{j} \rangle, Y_i,Y_{j})$ for all $i \neq j$.

If $j \not= i\pm 1$ (modulo $n$) then $Y_i$ and $Y_j$ commute thus 
$\varepsilon(\langle Y_i,Y_j \rangle, Y_i,Y_j) =0$. 

If $j = i+1$ for $i=1,\dots n-2$ then the group generated by $Y_i$ and $Y_{i+1}$ is the Heisenberg group therefore  
$\varepsilon(\langle Y_i,Y_{i+1} \rangle, Y_i,Y_{i+1})  \leq \sqrt{1/p} $, by Theorem~\ref{thm:rep-angle-improved} (see also~\cite[\S 4]{EJ}).

In the case $i=n-1$ and $j=n$, similar (but easier) computations as in the proof of Proposition~\ref{prop:nilpotent} show that the group generated by $Y_{n-1}$ and $Y_n$ is
nilpotent of class $2$ with abelianization $R^3$ and commutator subgroup is $\Trans {n-1}11 \Trans {n-1}12 \cong R^2$, thus 
$\varepsilon(\langle Y_{n-1},Y_n \rangle, Y_{n-1},Y_n)  \leq \sqrt{1/m(R)}$, where $m(R)$ is the smallest index of a proper ideal in $R$, see \cite[Cor.~4.4]{ER18}. By the definition of $p$, we have $m(R) \geq p$. Therefore $\varepsilon(\langle Y_{n-1},Y_n \rangle, Y_{n-1},Y_n)  \leq \sqrt{1/p}$.

The final case is $i=1$ and $j=n$, in this case the group 
generated by $Y_1$ and $Y_n$ is nilpotent of class $3$ with abelianization $R^3$ and commutator subgroup $\Trans n21 \Trans n22 \dTrans n1211 \cong R^3$, and center $\Trans n21 \Trans n22\cong R^2$. Set $H = \langle \langle Y_1,Y_n \rangle$. We claim that 
$$\varepsilon(H, Y_1,Y_n)  \leq \sqrt{2/p}.$$ 
This can be deduced from Theorem~\ref{thm:rep-angle-improved} as follows.

Set $Y = Y_1$, $X_0 = \Trans n12$ and $X'_0 = \Trans n11$, so that $Y_n = X_0 X'_0$. Let also $X = \Trans n12 \Trans n22 \dTrans n1211 \cong R^3$ and $X' = \Trans n11 \Trans n21 \cong R^2$. By Proposition~\ref{prop:nilpotent}, we have $\langle X_0, Y\rangle \cong \Gamma_{2, R}$ and $\langle X'_0, Y\rangle \cong \Gamma_{1, R}$. Moreover $X X'$ is an abelian normal subgroup of $H$. Therefore, given a unitary representation $(\pi, V)$ of $H$ without any non-zero invariant vectors, the space $V$ has an $H$-invariant decomposition  as a direct sum
$$V = \big(V^X \cap V^{X'}\big) \oplus \big(V^X \cap (V^{X'})^\perp\big) \oplus \big((V^X)^\perp \cap V^{X'}\big)\oplus \big((V^X)^\perp \cap (V^{X'})^\perp\big).$$
To evaluate the representation angle, it therefore suffices to treat one summand at a time. Since $X$ or $X'$ or both act trivially on each of the first three summands, the $H$-action factors through $\langle X'_0, Y\rangle$ or $ \langle X_0, Y\rangle$, so the desired bound directly follows from Theorem~\ref{thm:rep-angle-improved} in each of those cases. To finish the proof of the claim, we may therefore assume that neither $X$ nor $X'$ have a non-zero fixed vector on $V$. In particular $\langle X_0, Y\rangle$ does not have any non-zero fixed vector on $V$, so that $$\cos\big(\sphericalangle(V^{X_0}, V^{Y})\big) \leq \sqrt{2/p}$$
by Theorem~\ref{thm:rep-angle-improved}. The claim follows since  $V^{Y_n} = V^{ X_0   X'_0} \subseteq V^{X_0}$

%A similar argument as in  the proof of Theorem~\ref{thm:rep-angle-improved} gives that $\varepsilon(\langle Y_1,Y_n \rangle, Y_1,Y_n)  \leq \sqrt{2/p}$. (Alternatively, we have the bound $\sqrt{1/p}$ for the quotient of that group modulo the centre by \cite[Cor.~4.4]{ER18}, which then yields the slightly weaker bound 
%$$\varepsilon(\langle Y_1,Y_n \rangle, Y_1,Y_n)  \leq \sqrt{\sqrt{\frac 1 p} + \frac 1 p} \leq \frac{\sqrt{2}}{\sqrt[4] p}$$
%by \cite[Th.~4.1]{ER18}.)

These bounds can be plugged into the main result in~\cite{K}, and we need to verify that the symmetric matrix $A$ with $A_{ii}=1$, $A_{i, i+1} = A_{i+1, i} = -1/\sqrt{p}$ and $A_{1, n} = A_{n, 1} = -\sqrt{2/p}$  is positive definite. This is indeed the case by Proposition~\ref{prop:pos-def-matrix}(i) as soon as $p \geq 7$. For $p=5$ the latter criterion does not apply, but if in addition  $n \geq 4$, then we may invoke Proposition~\ref{prop:pos-def-matrix}(iii)%
\footnote{Alternatively we can numerically compute the eigenvalues of the resulting $4 \times 4$ matrix and see that the smallest one is approximately $0.0037$.}
since $1/\sqrt 5 \approx 0.447$ and $(1-1/\sqrt 5)(1-\sqrt{2/5}) \approx 0.203> 1/5 = (1/\sqrt 5)^2$. This ensures that $A$ is positive definite for all $n \geq 4$ and  $p \geq 5$. 
This implies that $v$ is close to an invariant vector for the whole group $G_{n,R}$, which implies that $G_{n,R}$ has property (T).
\end{proof}

A careful tracking of all constants involved in the above argument shows that if $S''$ is a generating set of $\EL_3(R)$ and $30$ is inverible in $R$, then
$$
\kappa\left(G_{3,R}, S'' \cup \{\trans 1221\}\right) \geq \frac{1}{20}\kappa(\EL_3(R), S'').
$$

\section{Constructing polynomial transvections}

We shall now see that the group  $G_{R,\mathbf{e}}$ is  quite large as soon as $\max \{e_i\} \geq 2$. The following result ensures that this group contains the element 
$\trans ijtr$ for all integers $t$ satisfying some congruence condition.

\begin{thm}
\label{thm:transvections}
Let $n \geq 3$, let $\mathbf{e}=(e_1, \dots, e_n)$ be a tuple of possitive integers and set $c = \max\{e_1, \dots, e_n\}$. Let also $R$ be a commutative unital ring such that $c!$ is invertible in $R$. 
If
$$  E = e_1e_2\dots e_n \geq 2,$$ 
then, for each integer $m \geq 0$, the group
$G_{R,\mathbf{e}}$   contains the elements
$$
\trans ijtr \quad \mbox{with }t =t_{i,j}  +m (E-1),
$$
where $t_{i,j}= e_ie_{i+1}\dots e_{j-1}$ and indices are ordered cyclically, and taken modulo $n$.
\end{thm}

The following consequence is immediate. 

\begin{coro}\label{cor:PolynomialTransvection}
Retain the hypotheses of Theorem~\ref{thm:transvections}. Then for each polynomial $P \in R[y]$, the group $G_{R,\mathbf{e}}$   contains the element
$\beta_{i;j}^P$ defined by the assignments
$$\left\{
\begin{array}{rcl}
    x_i & \mapsto & x_i + x_j^{t_{i, j}} P(x_j^{E-1})  \\
    x_\ell & \mapsto & x_\ell \qquad \text{for all } \ell \neq i.
\end{array}
\right.$$
\end{coro}

Those polynomial transvections $\beta_{i;j}^P$ will play a crucial role in the next section. 

\begin{coro}\label{cor:ElementaryAbelian}
Let $n \geq 3$, let $e_1, \dots, e_n \geq 1$ be integers and $p$ be a prime with $p > \max\{e_1, \dots, e_n\}$. Let also $R$ be an $\F_p$-algebra. If $ \max\{e_1, \dots, e_n\} > 1$, then $G_{R,\mathbf{e}}$   contains elementary abelian $p$-groups of infinte  rank. 
\end{coro}
\begin{proof}
This follows from Theorem~\ref{thm:transvections}. Indeed,  for fixed $i, j$, the group generated by 
$
\trans ijtr \quad \mbox{with }t =t_{i,j}  +m (E-1),
$
over all $m$ and all $r \in \mathbf F_p$ is elementary abelian of infinite rank.
\end{proof}

\begin{remark}
In the special case where $n=3$ and $e_i\leq 2$ for all $i$ and $R =\F_p$, these groups are quotients of KMS groups studied in~\cite[\S 7]{CCKW}. The corresponding epimorphisms for a couple of specific examples were described in \S\ref{sec:Hyp-Kazhdan}.  If $ \max\{e_1, e_2,e_3\} > 1$, it is shown in \cite[\S 7]{CCKW} that those KMS groups are hyperbolic. Corollary~\ref{cor:ElementaryAbelian} implies that the natural quotient map we have just mentioned cannot be injective. Indeed, in a hyperbolic group, the finite subgroups fall into finitely many conjugacy classes (see~\cite{Brady}). In particular the order of a finite subgroup is bounded above.
\end{remark}

\begin{proof}[Proof of Theorem~\ref{thm:transvections}]
We divide the proof into several steps. 

\begin{step}\label{st:1}
The theorem holds for $m=0$ and $j=i+1$.
\end{step}

Indeed, this is clear since the required elements are among the generators of $G$ by definition.

\begin{step}\label{st:2}
The theorem holds for $m=0$ and all indices $i \neq j$.
\end{step}

We fix the index $i$. By Step~\ref{st:1} we may assume that $j>i+1$ and proceed by induction on $j-i$.
By the induction hypothesis, the group $G$ contain the elements $\trans {i+1}j{t_{i+1,j}} r $.  We invoke Proposition~\ref{prop:nilpotent} applied to triple of indices $i, i+1, j$, and with the exponents $(c, d) = (e_i, t_{i+1, j})$. Since every positive integer $\leq e_i$ is invertible in $R$ by hypothesis, we deduce from the proposition that $G$ contains $\dtrans i{i+1}j{e_i-\ell}{\ell t_{i+1,j}}r$ for all $r \in R$ and all $\ell \in \{0, 1, \dots, e_i\}$. Taking $\ell = e_i$, we deduce that $G$  contains  
$\trans ij{e_i t_{i+1,j}}r = \trans ij{t_{i,j}}r $. This  completes the induction step. 

\begin{step}\label{st:3}
For all indices $i$, the group $G$ contains $\dTrans i{i+1}{i+2}{e_i-1}{e_{i+1}}$ and $\dTrans i{i+1}j1{(e_i -1)t_{i+1, j}}$. 
\end{step}

This follows by taking $\ell = 1$ and $\ell = e_i-1$ in the proof of the previous step.

\begin{step}\label{st:4}
The theorem holds for $m=1$ and $j=i+1$.
\end{step}

In view of Step~\ref{st:2}, we know that  $G$  contains 
$\Trans {i+2}{i+1}{t_{i+2,i+1}}$. Moreover, by Step~\ref{st:3} we also know that  $G$  contains    $\dTrans i{i+1}{i+2}{e_i-1}{e_{i+1}} = \dTrans i{i+2}{i+1}{e_{i+1}}{e_i-1}$.
Notice that $t_{i+2,i+1} = E/e_{i+1}$. Invoking  Proposition~\ref{prop:nilpotent:variant} to the triple of indices $i, i+2, i+1$ and the exponents $(c, d, e) = (e_{i+1}, e_i-1, E/e_{i+1})$, we deduce that $G$ contains the subgroup
$\dTrans i{i+2}{i+1}{e_{i+1}-\ell}{e_i-1+\ell E/e_{i+1}}$.

Taking $\ell = e_{i+1}$, it follows  that $G$ contains  
$\trans i{i+1}{C_i}r$ for all $r \in R$, where 
$$
C_i = e_i + (E-1). 
$$
This finishes the proof of Step~\ref{st:4}. 

\begin{step}\label{st:5}
For all $i$, the group $G$ contains 
$\dTrans i{i+1}{i+2}{e_i - E/e_{i+1} + E - 1}1 = \dTrans i{i+2}{i+1}1{e_i - E/e_{i+1} + E - 1}$.
\end{step}
This follows by taking $\ell = e_{i+1}-1$ in the proof of the previous step.

\begin{step}                \label{st:6}
If for some $m$ and some indices $i, j$ with $j \not \in \{i-1, i\}$, the group $G$ contains $\Trans ij{t_{i, j}+ m(E-1)}$, then $G$ also contains $\Trans {i-1}j{t_{i-1, j}+ m(E-1)}$.
\end{step}

Recall from Step~\ref{st:3} that $G$ contains $\dTrans {i-1}{i}j1{(e_{i-1} -1)t_{i, j}}$.  We may thus invoke  Proposition~\ref{prop:nilpotent:variant} to the triple of indices $(i-1, i, j)$ and the exponents 
$$(c, d, e) = \big(1, (e_{i-1} - 1)t_{i, j}, t_{i, j} + m(E-1)\big).$$ The required assertion follows.

\begin{step}                \label{st:7}
If for some $m$, the group $G$ contains $\Trans ij{t_{i, j}+ m(E-1)}$ for all indices $i \neq j$, then  $G$ also contains $\Trans i{i+1}{e_i+ (m+1)(E-1)}$.
\end{step}

Recall that $E/e_{i+1} = t_{i+2, i+1}$. 
Recall also from Step~\ref{st:5} that $G$ contains the subgroup $\dTrans i{i+2}{i+1}1{e_i-E/e_{i+1}+E-1}$.  We may thus invoke Proposition~\ref{prop:nilpotent:variant} to the indices $(i, i+2, i+1)$ and the exponents 
$$(c, d, e) = \big(1, e_i-E/e_{i+1} + E-1, E/e_{i+1}+m(E-1)\big).$$ 
It follows that $G$ also contains $\Trans i{i+1}{e_i+(m +1)(E-1)}$, as required.

\medskip

To finish the proof, we  use a double induction on $m$ and on $j-i$, applying alternatively Step~\ref{st:6} and~\ref{st:7}. The conclusion of each of those ensures that the hypothesis of the other is satisfied. The conclusion of the theorem follows.  
\end{proof}

\begin{remark}
The conclusions of Theorem~\ref{thm:transvections} can be strenghtened: indeed, adapting the proof, one can show that  $G_{R,\mathbf{e}}$ also contains elements 
$\dTrans ijk1t$ for any triples of distinct indices $i,j,k$ provided that $t$ satisfies some congruence condition modulo $E-1$. 
\end{remark}

The following definition will play an important role in the next section. 

\begin{definition}
\label{grading}
The ring  $R_n=R[x_1,\dots, x_n]$  has a \textbf{grading} by the cyclic group $\Z/(E-1)\Z$, defined as follows: for each $i \in \{1,\dots,n\}$, the degree of $x_i$ is  $t_{i, 1} = \prod_{j \geq i}^n e_j$. 
Let $R_n^{(s)}$ denote the homogeneous component of $R_n$ of degree $s$, clearly $R_n^{(0)}$ is a subring of $R$, containing, $x_i^{E-1}$, but it also contains elements like $x_i^tx_j$ when $t$ satisfies some congruence restriction mod $E-1$. Observe that each homogeneous component $R_n^{(s)}$ is a module over $R_n^{(0)}$.
\end{definition}

\begin{observation}
The action of the generators of $G$ preserve the grading of $R_n$ mentioned above, because the degree  of $x_i$ is the same as the degree of $x_{i+1}^{e_i}$, so that   each generator of $G$ indeed preserves the grading.  This implies that the modular restriction on $t$ in Theorem~\ref{thm:transvections} cannot be removed.
\end{observation}

\begin{remark}
\label{rem:more-transvections}
The result can be extended to show that the group $G$ contain many automorphisms  $\alpha_{i,f}$ mapping  $\alpha_{i,f}(x_i)= x_i +f$ where $f$ is some polynomial on the other variables, and $\alpha_{i,f}(x_\ell)=x_\ell$ for all $\ell \neq i$. However, we cannot get all such automorphisms: first there is some modular restriction on degrees of the monomials appearing in $f$, which is necessary in order for $\alpha_{i,f}$ to preserve the grading of $R_n$. 

The other issue is that Proposition~\ref{prop:nilpotent:variant} requires that every prime smaller than or equal to the exponent $c$ be invertible in $R$. This  forces that any monomial in $f$ to contain a variable of degree at most $c$. We do not know if this condition can be removed, equivalently we do not know if $G$
coincides with the tame automorphism group of the graded ring $R_n$, however it is clear that  $G$ is a very ``large'' subgroup of the group of tame automorphisms of the graded ring $R_n$. 
\end{remark}

\section{Constructing finite quotients}

\subsection{Action on  affine spaces}
From now on, we shall focus on the case $R = \mathbf F_p$. 
Our next goal is to construct finite quotients of the group $G_{\F_p,\mathbf{e}}$. Those  naturally arise as finite quotients of $\Aut(\F_p[x_1,\dots,x_n])$, by considering the spectrum of the ring $\F_p[x_1,\dots,x_n]$ as the points of an affine scheme which is simply  the $n$-dimensional affine space over the prime field $\F_p$. That viewpoint suggests the following construction. 

Given a commutative $\F_p$-algebra $A$, the set $A^n$ may be identified with the set of $\F_p$-algebra homomorphisms
$\Hom(\F_p[x_1,\dots,x_n], A)$. To an element $(a_1, \dots, a_n)$, one associates the evaluation map 
$$
f \in\F_p[x_1,\dots,x_n] \mapsto f(a_1, \dots, a_n) \in A.
$$
and to a homomorphism $\phi \in \Hom(\F_p[x_1,\dots,x_n], A)$, one associates the $n$-tuple $\big(\phi(x_1), \dots, \phi(x_n)\big)$. 
The group $\Aut(\F_p[x_1,\dots,x_n])$  acts on $\Hom(\F_p[x_1,\dots,x_n], A)$ by pre-composition: each automorphism $\alpha \in \Aut(\F_p[x_1,\dots,x_n])$ yields the map 
$$
\Hom(\F_p[x_1,\dots,x_n], A) \to \Hom(\F_p[x_1,\dots,x_n], A) : \phi \mapsto \phi \circ \alpha^{-1}.
$$
It is straightforward to check that this defines indeed a permutation action of the group $\Aut(\F_p[x_1,\dots,x_n])$. The group $\Aut(A)$ also acts on 
$\Hom(\F_p[x_1,\dots,x_n], A)$ by post-composition. Clearly the actions of $\Aut(A)$ and $\Aut(\F_p[x_1,\dots,x_n])$ commute. 

Using the natural  bijection $A^n \to \Hom(\F_p[x_1,\dots,x_n], A)$ recalled above, we see that an automorphism $\alpha \in \Aut(\F_p[x_1,\dots,x_n])$ acts on $A^n$ via the map
$$
A^n \to A^n : (a_1, \dots, a_n) \mapsto \big(\alpha^{-1}(x_1)(a_1, \dots, a_n), \dots, \alpha^{-1}(x_n)(a_1, \dots, a_n) \big).
$$
We will denote with $e_A:\Aut(\F_p[x_1,\dots,x_n]) \to \Sym(A^n)$
the resulting homomorphism to the symmetric groups of the set $A^n$.

By letting $A$ vary over the collection of finite-dimensional $\F_p$-algebras, we obtain numerous finite quotients of $\Aut(\F_p[x_1,\dots,x_n])$. 
By letting $A$ run over all finite field extensions of $\F_p$ and using  Hilbert’s Nullstellensatz one recovers the well known fact.

\begin{prop}\label{prop:RF}
The group $\Aut(\F_p[x_1,\dots,x_n])$ is residually finite.
\end{prop}

In particular, the group $G = G_{\F_p,\mathbf{e}}$ is residually finite. For suitable $n$-tuples $\mathbf{e}$, this can actually be strengthened as follows. 

\begin{prop}\label{prop:Res-p}
Let  $k$ be a finite extension of $\F_p$, and let $\mathbf e = (e_1, \dots, e_n)$ be an $n$-tuple with $e_i \geq 1$ for all $i$.  If $\max\{e_1, \dots, e_n\} \geq 2$, 
then $G_{k,\mathbf{e}}$ is residually-$p$. 
\end{prop}
\begin{proof}
Let $\mathfrak m^0 = k[x_1, \dots, x_n]$, let $\mathfrak m$ be the ideal consisting of those polynomials with zero constant term, and define inductively $\mathfrak m^{d+1}$ as the ideal generated by $\{fg \mid f \in \mathfrak m^d, \ g \in \mathfrak m\}$. Clearly, the ideal $\mathfrak m^d$ is invariant under $\Aut(k[x_1,\dots,x_n])$; this yields a group homomorphism from $\Aut(k[x_1,\dots,x_n])$ to the automorphism group of the quotient algebra $k[x_1,\dots,x_n]/\mathfrak m^{d+1}$. The kernel of that homomorphism is denoted by $\Aut_d(k[x_1,\dots,x_n])$.

Let $k\llbracket x_1, \dots, x_n\rrbracket$ denote the algebra of formal power series in the indeterminates $x_1, \dots, x_n$ with coefficients in $k$. The natural embedding $k[x_1, \dots, x_n] \to k\llbracket x_1, \dots, x_n\rrbracket$ yields an injective homomorphism of $\Aut_i(k[x_1,\dots,x_n])$ into the group $\Aut_i(k\llbracket x_1,\dots,x_n\rrbracket)$ for all $i \geq 0$. It is known that $\Aut_1(k\llbracket x_1,\dots,x_n\rrbracket)$ is a pro-$p$ group (see \S 8.5 in \cite{Shalev} and \S2 in \cite{Veronelli}), so that $\Aut_1(k[x_1,\dots,x_n])$ is residually-$p$. Observe that $G = G_{k,\mathbf{e}}$ is a subgroup of $\Aut_1(k[x_1,\dots,x_n])$ if and only if $e_i \geq 2$ for all $i$. Here, the hypothesis only ensures that $e_i \geq 2$ for some $i$. Without loss of generality, we may assyme that $e_n \geq 2$.  To finish the proof, it suffices to show that  the image of $G$ under  the quotient map  $\phi \colon \Aut(k[x_1,\dots,x_n]) \to \Aut(k[x_1,\dots,x_n])/\Aut_1(k[x_1,\dots,x_n])$ is a finite $p$-group. Since $e_n \geq 2$, the image of $\tau_n(r) = \trans n{1}{e_n}r$  under $\phi$ is trivial.  Therefore, the image of 
$G   = \left\langle 
\trans i{i+1}{e_i}r \, \big| \,i=1, 2, \dots, n;\  r\in k 
\right\rangle$  under $\phi$ coincides with the image of $\left\langle 
\trans i{i+1}{e_i}r \, \big| \,i=1, 2, \dots, n-1;\  r\in k 
\right\rangle$. The latter group is a subgroup of $\mathrm{SL}_n(k)$ consisting of upper unitriangular matrices, and is therefore a finite $p$-group. Hence $\phi(G)$ is a finite $p$-group as well. 
\end{proof}

As mentioned in the introduction, the group $G_{\F_p,(1, 1, \dots, 1)}$  of rank~$n \geq 3$ is isomorphic to $\mathrm{SL}_n(\F_p)$, so the condition that $\max\{e_i\} \geq 2 $ cannot be removed in Proposition~\ref{prop:Res-p}. We also remark that if $p > \max\{e_1, \dots, e_n\}$, then $G_{\F_p, \mathbf e}$ contains an elementary abelian $p$-group of infinite rank by Corollary~\ref{cor:ElementaryAbelian}, so it is not virtually residually-$q$ for any prime $q \neq p$. 

To obtain a more precise description of specific finite quotients of the group $G = G_{\F_p,\mathbf{e}}$, we will use the  $\Aut(\F_p[x_1,\dots,x_n])$-action on $A^n$ mentioned above, in the  special case where the algebra $A$ is a finite field extension of $\F_p$. In other words 
$A$ is a finite field of order $q = p^s$ for some $s \geq 1$. 
The previous proposition actually gives that the group  $\Aut(\F_p[x_1,\dots,x_n])$ embeds into the product of $\prod_\ell \Sym_{p^{n\ell}}$ via the product of the maps $e_{\F_{p^\ell}}$. 
The images of $G_{\F_p,\mathbf{e}}$ under the homomorphisms $e_{\F_{p^\ell}}$ are not the full symmetric groups. 
Indeed, these groups are generated by elements of odd prime order $p$, so any permutation action of $G$ on a finite set consists of even permutations.
Our main results, Theorems~\ref{thm:transitivity}, \ref{thm:almost-k-transitivity} and Corollary~\ref{cor:Alt-quotient}, provide more information about these images. 
The proofs of  these results require some preparation. We will start with a small example which conveys the main idea but avoids most technical difficulties. 

\subsection{The prime field case}
Before considering the general case, we focus on the special case when $A = \F_p$ and the base field $R= \F_p$ does not contain any non-trivial $E-1$ roots of $1$. The latter obviously holds  for example when $E=2$.

\begin{thm}
\label{thm:k-transitivity-base-field}
Let $n \geq 3$ and $e_1, \dots, e_n \geq 1$ be integers such that $E = e_1 \dots e_n \geq 2$. Let $p$ be a prime with $p > \max \{e_1, \dots, e_n\}$, and such that the only root of $x^{E-1} =1$ in $\F_p$ is the trivial root $x=1$. Let $G = G_{\F_p,\mathbf{e}}$, where $\mathbf e =(e_1, \dots, e_n)$. 

Then the $G$-action on $\F_p^n$  fixes the point $0$ and acts $(p-1)$-transitively on $\F_p^n \setminus \{0\}$.
\end{thm}
\begin{proof}
Let $k\geq 1$ be an integer with $k \leq p-1$. We fix  a  standard $k$-tuple of vectors, denoted  $(\sigma_1, \dots, \sigma_k)$ and defined by 
$$\sigma_i =(i, 0, \dots, 0) \in \F_p^n.$$
Let now  $\phi_1,\dots, \phi_k$ be distinct elements in $\F_p^n \setminus \{0\}$. We must find an element of $G$ mapping that $k$-tuple to the standard one.  
We proceed in several steps. As we shall see, the key is to combine repeatedly  Lemma~\ref{lem:polynomial-map-k-tuple} with Corollary~\ref{cor:PolynomialTransvection}. Let us briefly discuss the respective assumptions of those results. 

We shall apply Lemma~\ref{lem:polynomial-map-k-tuple} to $k$-tuples of the form $(\mu_1^{E-1}, \dots, \mu_k^{E-1})$, with $\mu_i \in \F_p^*$. The hypothesis made on $\F_p$ ensures that the map $\F_p \to \F_p : x \mapsto x^{E-1} $ is injective. Hence if the $\mu_i$'s are pairwise distinct, then so are their $(E-1)^{\text{st}}$ powers. In particular those elements have pairwise distinct minimal polynomials over $\F_p$. This will ensure that the assumptions of Lemma~\ref{lem:polynomial-map-k-tuple} are fulfilled. By hypothesis, the assumptions of Theorem~\ref{thm:transvections} are equally satisfied. 

For $j \in \{1, \dots, n\}$, we denote the $j^{\text{th}}$ coordinate of an element $\psi \in \F_p^n$ by $\psi(j) \in \F_p$.

\setcounter{step}{0}
\begin{step}\label{st:f_p:1}
There exists $g \in G$ such that $g\phi_1(n) \neq 0$. 
\end{step}

We may assume that $\phi_1(n)=0$. 
Since the vector $\phi_1$ is non-zero, there exists $i \in \{1, \dots, n-1\}$ with $\phi_1(i) \neq 0$. By Theorem~\ref{thm:transvections}, we have $\trans nit1 \in G$ for some integer $t \geq 1$. The required assertion holds with $g = \trans nit1$.

\begin{step}\label{st:f_p:2}
Let $s \in \{1, \dots, k\}$. Assume that for some $\ell \in \{1, \dots, n\}$, the elements $\phi_1(\ell), \phi_2(\ell), \dots, \phi_s(\ell)$ are non-zero and pairwise distinct. Then there exists $g \in G$ such that $g\phi_i(\ell) = \phi_i(\ell)$  and $g\phi_i(j)=0$ for  all $i \in \{1, \dots, s\}$ and all $j \in \{1, \dots, n\} \setminus \{\ell\}$.
\end{step}

Let us first fix $j \in \{1, \dots, n\} \setminus \{\ell\}$. 
We use the notation $t_{j, \ell} = e_j \dots e_{\ell-1}$ from Theorem~\ref{thm:transvections}.  
Applying Lemma~\ref{lem:polynomial-map-k-tuple} to the $s$-tuples $(\mu_1, \dots, \mu_s)$  and $(\nu_1, \dots, \nu_s)$ defined by $\mu_i = \phi_i(\ell)^{E-1}$ and $$\nu_i = \frac{-\phi_i(j)}{\phi_i(\ell)^{t_{j, \ell}}},$$
we obtain a polynomial $f \in \F_p[x]$ with $f(\mu_i) = \nu_i$ for all $i=1, \dots, s$. We then invoke Corollary~\ref{cor:PolynomialTransvection}, ensuring the existence of a polynomial transvection  $g_j \in G$ fixing  the indeterminate $x_m$ for all indices $m \neq j$ and mapping $x_j$ to $x_j + x_\ell^{t_{j, \ell}} f(x_\ell^{E-1})$. By construction, we have  $g_j\phi_i(j) = 0$ for all $i \in \{1, \dots, s\}$.

Doing this for all $j \in \{1, \dots, n\} \setminus \{\ell\}$, we obtain elements $g_j$ that commute pairwise.   It follows that the required assertion holds with $g = \prod_{j \neq \ell} g_j$.

\begin{step}\label{st:f_p:3}
Let $s \in \{1, \dots, k-1\}$. Assume that for some $\ell \in \{1, \dots, n\}$, the elements $\phi_1(\ell), \phi_2(\ell), \dots, \phi_s(\ell)$ are non-zero and pairwise distinct. Then there exists $g \in G$ such that $g\phi_i(\ell)$ are non-zero and pairwise distinct for  all $i \in \{1, \dots, s, s+1\}$, and that $g\phi_i(j)=0$ for  all $i \in \{1, \dots, s, s+1\}$ and all $j \in \{1, \dots, n\}\setminus \{\ell\}$. 
\end{step}

We first invoke the previous step. We may thus assume that $\phi_i(j)=0$ for  all $i \in \{1, \dots, s\}$ and all $j \in \{1, \dots, n\}\setminus \{\ell\}$.  By hypothesis $\phi_{s+1}$ is non-zero and distinct from $\phi_1, \dots, \phi_s$. Therefore, if $\phi_{s+1}(j)=0$ for all  $j \neq \ell$, we are already done. Otherwise, we may assume that $\phi_{s+1}(j)\neq 0$ for some $j \neq \ell$. Let now $\nu$ be any non-zero element of $\F_p$ different from  $\phi_1(\ell), \phi_2(\ell), \dots, \phi_s(\ell)$. Let $f \in \F_p[x]$ be the polynomial defined by $f(x) = x + \frac{\nu - \phi_{s+1}(\ell)}{\mu^{t_{\ell, j}}} - \mu^{E-1}$, where $\mu = \phi_{s+1}(j)$. Hence we have  $f(\mu^{E-1})= \frac{\nu - \phi_{s+1}(\ell)}{\mu^{t_{\ell, j}}}$.   We then invoke Corollary~\ref{cor:PolynomialTransvection}, ensuring the existence of an element $h \in G$  of the form $h = \trans \ell j{t_{\ell, j}}r$ such that $h\phi_i(\ell) = \phi_i(\ell)$  for all $i=1, \dots, s$ and $h\phi_{s+1}(\ell) = \nu$.  We finish by invoking again the previous step. 

\begin{step}
End of the proof. 
\end{step}

By Step~\ref{st:f_p:1} we may assume that $\phi_1(n) \neq 0$. Now we use induction on $s$ and Step~\ref{st:f_p:3} with $\ell = n$. This proves that, after transforming by some element of $g \in G$, we may assume that the elements $\phi_1(n), \phi_2(n), \dots, \phi_k(n)$ are non-zero and pairwise distinct, and moreover $\phi_i(j) = 0$ for all $i=1, \dots, k$ and all $j < n$. 

We then invoke Lemma~\ref{lem:polynomial-map-k-tuple} to the $k$-tuples $(\mu_1, \dots, \mu_k)$  and $(\nu_1, \dots, \nu_k)$ defined by $\mu_i = \phi_i(n)^{E-1}$ and 
$$\nu_i = \frac{i}{\phi_i(n)^{t_{1, n}}}.$$ 
This yields  a polynomial $f \in \F_p[x]$ with $f(\mu_i) = \nu_i$ for all $i=1, \dots, s$. We then apply Corollary~\ref{cor:PolynomialTransvection} to construct a polynomial transvection $h \in G$ fixing  the indeterminate $x_m$ for all $m \neq 1$ and mapping $x_1$ to $x_1 + x_n^{t_{1, n}} f(x_n^{E-1})$. It follows that $h\phi_i(1) = i$ for all $i=1, \dots, k$. By applying Step~\ref{st:f_p:2} with $\ell = 1$, we may now apply another element $g \in G$ so that $gh\phi_i(1) = i$ and $gh\phi_i(j)= 0$ for all  $i \in \{1, \dots, k\}$ and $j \in \{2, \dots, n\}$. It follows that $gh\phi_i = \sigma_i$ for all $i$, and we are done. 
\end{proof}

\begin{coro}\label{cor:Alt-base-field}
Let $n \geq 3$ and $e_1, \dots, e_n \geq 1$ be integers such that $E = e_1 \dots e_n \geq 2$. Let $p$ be a prime with $p > E$
and such that $p-1$ and $E-1$ are relatively prime. 

Then $G = G_{\F_p,\mathbf{e}}$ maps  onto  $\Alt(p^n-1)$.
\end{coro}
\begin{proof}
By Theorem~\ref{thm:k-transitivity-base-field},  the $G$-action on the set  $\F_p^n \setminus \{0\}$  of cardinality $p^n-1$ is $(p-1)$-transitive. 

Assume first that $p \geq 5$. Hence the $G$-action is $4$-transitive.
Recall that a finite $4$-transitive group on a set of cardinality
~$\geq 25$ is the full alternating or symmetric group on that set (see~\cite[Th.~4.11]{Cameron}). By hypothesis, we have  $p^n-1  \geq 25 $. Moreover 
$G_{\F_p,\mathbf{e}}$ does not have any quotient isomorphic to $\Sym(n)$ for $n \geq 2$ because it is generated by elements of odd %prime 
order. The conclusion follows in this case. 

Assume now that $p =3$. By the above,  we know that the image of $G$ in $\Sym(\F_3^n \setminus \{0\})$, that we shall denote by $H$, is $2$-transitive.  Since $p> E \geq 2$, we have $E=2$. 
Without loss of generality, we may assume that $e_1=1$. By Corollary~\ref{cor:PolynomialTransvection}, there exists a polynomial transvection $g \in G$ acting on $\F_3^n$ as 
$$
g \colon (a_1, \dots, a_n) \mapsto (a_1-a_2 + a_2^2, a_2, \dots, a_n).
$$
In particular, the only points that are not fixed by $g$ satisfy $a_2 = 2$. Thus $H$ contains a non-trivial permutation $h$ of $\F_3^n \setminus \{0\}$ fixing at least 
$$
3^n -1 - 3^{n-1}
$$
points. 
Now we invoke~\cite[Corollary~1]{GurMag}, which ensures that if a  primitive group $H$ of degree $d$ contains a non-trivial element fixing more than $\frac 4 7 d$ points, then its general Fitting subgroup $F^*(H)$ is a product of alternating groups. The proportion of fixed points of $h$ is 
$$
1- \frac{3^{n-1}}{3^n-1}=1-\frac{1}{3}\left(1+ \frac{1}{3^n-1}\right) = \frac{2}{3} - \frac{1}{3^{n+1}-3} > \frac{4}{7}
$$
since $n \geq 3$, so that  $F^*(H)$ is a product of alternating groups. By Burnside's theorem (see~\cite[Theorem~4.3]{Cameron}), a minimal normal subgroup of a $2$-transitive group is either elementary abelian (with a regular action) or non-abelian simple (with a primitive action). We infer that the the socle of $H$ is an alternating group. The conclusion now follows from the classification of the finite $2$-transitive groups, see \cite[\S7.4]{Cameron}. 
\end{proof}

\begin{remark}\label{rem:CFSG-free}
As mentioned in the introduction, for $p$ sufficiently large, the use of the CFSG via the classification of $4$-transitive groups in the proof of Corollary~\ref{cor:Alt-base-field} can be bypassed, using the main result of~\cite{Babai} or of~\cite{Pyber}, since the $G$-action on the set  $\F_p^n \setminus \{0\}$ is $(p-1)$-transitive by Theorem~\ref{thm:k-transitivity-base-field}. 
\end{remark}

\begin{coro}\label{cor:Z[1/30]}
For each $n \geq 3$, the group $G_{n,\Z[1/30]}$
defined in Theorem~\ref{thm:T-bis} has property~(T) and maps onto the alternating groups $\Alt(p^n-1)$ for all primes $p \geq 7$. 
The same holds for the group $G_{n,\Z[1/6]}$ for all $n \geq 4$ and all primes $p \geq 5$.
\end{coro}
\begin{proof}
The group $G_{n,R}$ contains $G_{R,(1,\dots, 1,2)}$ as a subgroup. Moreover, a surjective ring homomorphism  $R\to \F_p$ yields an action of   $G_{n,R}$ on the set $\F_p^n\setminus\{0\}$. Since $G_{R,(1,\dots, 1,2)}$ acts through the quotient map $G_{R,(1,\dots, 1,2)} \to G_{\F_p,(1,\dots,1,2)}$, the image of that action is the full alternating group by the previous corollary.
\end{proof}

\section{Alternating groups as expanders}

This section is devoted to proving Theorem~\ref{thm:expanders-intro}. We present it here as it only relies on the results obtained thus far, and does not require the technicalities we will deal with when extending Theorem~\ref{thm:k-transitivity-base-field} over extensions of the base field in the following section. 

For the reader's convenience, we reproduce the statement of the theorem here:

\begin{thm}\label{thm:expanders-core}
Let $p$  be an odd prime prime. 

\begin{enumerate}[label=(\roman*)]
    \item The permutations 
$$
\sigma(x,y,z) = (y,z,x)
\quad
\alpha(x,y,z) = (x+y,y,z)
\quad
\beta(x,y,z) = (x+y^2,y,z),
$$
acting on the set $\F_p^3 \setminus \{(0, 0, 0)\}$ of cardinality $p^3-1$, generate the full alternating group $\Alt(p^3-1)$. The associated Cayley graphs form expanders of degree~$6$. 

\item The permutations 
$$
\rho(x_1,x_2, x_3, x_4, x_5, x_6, x_7) = (x_2, x_3, x_4, x_5, x_6, x_7,x_1)
$$
$$
\gamma(x_1,x_2, x_3, x_4, x_5, x_6, x_7) = (x_1 + x_2 ,x_2, x_3, x_4+x_6^2, x_5, x_6, x_7)
$$
acting on the set $\F_p^7 \setminus \{(0, \dots, 0)\}$ of cardinality $p^7-1$, generate the full alternating group $\Alt(p^7-1)$. The associated Cayley graphs form expanders of degree~$4$. 
\end{enumerate}
\end{thm}

\begin{proof}%[Proof of Theorem~\ref{thm:expanders-intro}]
We start with (i). Let $H$ be the permutation group generated by $\sigma, \alpha, \beta$, which have order $3, p, p$ respectively. In particular   $\sigma, \alpha, \beta$ are even permutations.

Observe that the permutations  
$$\beta, \sigma\alpha \sigma^{-1},\sigma^{-1}\alpha\sigma \in H$$
coincide with  the  images of the generators $\tau_1(1), \tau_2(1)$ and $\tau_3(1)$ of the group $G_{\F_p, (2, 1, 1)}$ in its natural action on $\F_p^3 \setminus \{(0, 0, 0)\}$. Therefore, it follows from  Corollary~\ref{cor:Alt-base-field} that $H$ is the full   alternating group $\Alt(p^3-1)$.

%By definition its generators act as $\beta$, $\sigma\alpha \sigma^{-1}$ and $\sigma^{-1}\alpha\sigma$. This implies that $\langle Q \rangle = \Alt(p^3-1)$, where $Q = \{\sigma, \alpha, \beta\}$ (since all generators are even permutations).

Set $Q =  \{\sigma, \alpha, \beta\}$. To prove the second assertion, we need to show that the unitary representation $\pi$, defined as the subrepresentation of the regular representation of $\Alt(p^3-1)$ on the orthogonal of the constant functions, does not have $(Q, \varepsilon)$-invariant vectors, where $\varepsilon >0$  is smaller than some constant which is independent of $p$. Without loss of generality, we may assume that $p \geq 11$. 

We now assume that $\pi$ has a $(Q, \varepsilon)$-invariant unit vector $v$, for some $\varepsilon >0$. %Set $R = \mathbf Z[1/30]$. 
Observe that the subgroup $L$ of $H$ generated by $\alpha$, $\sigma\alpha \sigma^{-1}$ and $\sigma^{-1}\alpha\sigma$ is a copy of $\SL_3(\F_p)$. Hence we may view the restriction $\pi |_L$ as a representation of the group $\SL_3(\Z)$, which factors thorugh $\SL_3(\F_p)$. Using property (T) for the group $\SL_3(\Z)$ we deduce from Lemma~\ref{lem:almost-invariant} that $v$ is $(\SL(\F_p), D\varepsilon)$-invariant vector for some constant $D$ which depends only on the Kazhdan constant of $\SL_3(\Z)$. In particular $D$ is independent of $p$. Using that $\trans 322r = \left[\trans 311r, \trans 1221\right] $, it follows that the vector  $v$ is $(S, 3D\varepsilon)$-invariant, where $S$ is the generating set for  $G_{\F_p,( 2, 1, 1)}$ consisting of 
%all %the 
all those elements $\trans{11}{32}{2}{r},\trans{22}{13}{1}{s}$ and $\trans{3}{21}{21}{t}$, with $r, s, t \in \mathbf F_p$. We now invoke  %Theorem~\ref{thm:T-intro}
Lemma~\ref{lem:almost-invariant} and Theorem~\ref{thm:T-for-G}. Those ensure   that the unit vector $v$ is also a $(G_{\F_p, ( 2, 1, 1)}, C_p\varepsilon)$-invariant vector, where $C_p = \frac {6D} {\kappa_p} $ and $\kappa_p$ is the Kazhdan constant for $G_{\F_p, ( 2, 1, 1)}$ with respect to the generating set $S$. Notice that  Theorem~\ref{thm:T-for-G} provides a lower bound for $\kappa_p$ which is a strictly increasing function of $p$. In particular it is bounded below by the value of that function at $p=11$.  Therefore, we deduce that $v$ is a $(G_{\F_p, ( 2, 1, 1)}, C\varepsilon)$-invariant vector, where $C $ is now   independent of $p$.  In the case where  $C\varepsilon < 1$, we deduce that $\pi$ contains a non-zero vector that it invariant under the whole group $G_{\F_p,( 2, 1, 1)}$. This is impossible since the latter group maps onto $\Alt(p^3-1)$ by Corollary~\ref{cor:Alt-base-field}. 
Thus we have reached contradiction by choosing $\varepsilon<1/C$. This finishes the proof of (i).

\medskip

%It follows that the vector  $v$ is $(S, 3D\varepsilon)$-invariant, where $S$ the generating set for  $G_{\F_p, 3 ; 2, 1, 1}$ consisting of a the powers of $\trans{1}{2}{2}{1},\trans{2}{3}{1}{1}$ and $\trans{3}{1}{1}{1}$. We now invoke Theorem~\ref{thm:T-intro}, to get that $v$ is also $(G_{\F_p, 3 ; 2, 1, 1}, C\varepsilon)$-invariant vector for some $C$. However, by Corollary~\ref{cor:Alt-base-field} this is not possible if $C\varepsilon < 1$ since $G_{\F_p, 3 ; 2, 1, 1}$ maps onto $\Alt(p^3-1)$. Thus we have that $\varepsilon > 1/C$ which finishes the proof.

The proof of part (ii) is similar. We notice that the commutator $\tau=[\gamma,\rho \gamma \rho^{-1}]$
acts as
$$
\tau(x_1,x_2, x_3, x_4, x_5, x_6, x_7) = (x_1+x_3,x_2, x_3, x_4, x_5, x_6, x_7)
$$
thus $\tau$ and $\rho$ generate a copy of the group $\SL_7(\F_p)$. This implies that there is a  word $w$ 
in $\rho$ and $\gamma$ which acts $\F_p^7 \setminus \{0\}$ as
$$
w(x_1,x_2, x_3, x_4, x_5, x_6, x_7) = (x_1,x_2, x_3, x_4+x_6^2, x_5, x_6, x_7),
$$
such that the word $w$ is \textit{short} in the sense that its  length  is bounded above independently of $p$. Indeed,  one may  take $w = \gamma \tau^{-g}$, where $g$ acts on the coordinates as the permutation $(12)(34)$. Such an element $g$  exists and can be expressed as a short   word in  $\rho$ and $\gamma$  since $\SL_7(\Z)$ contains all even permutations of the variables.
Furthermore by conjugating $w$ by a suitable permutation of the variables, there exits another short word which acts as $\trans{1}{2}{2}{1}$.  Therefore the image of $\langle \rho, \gamma\rangle$ in $\Sym(p^7-1)$ contains the image of  $G_{\F_p,( 2, 1, \dots,  1)}$. In particular its image is the full alternating group $\Alt(p^7-1)$ by Corollary~\ref{cor:Alt-base-field} (since $\rho$ is of order $7$ and $\gamma$ is of order $p$, both elements are even permutations)

To finish the proof, we repeat  the argument from part (i) and see that the Cayley graphs of $\Alt(p^7-1)$ with the generating set $\{\rho, \tau, w\}$ are expanders. Since these elements can be expressed as short words in $\gamma$ and $\rho$, we infer  that the Cayley graphs of $\Alt(p^7-1)$ with respect to 
$\{ \gamma, \rho\}$ are also expanders.
\end{proof}
\begin{remark}
It is possible to track the constants in the argument above and obtain bounds for the spectral gap of the resulting expander graphs. A quick computation bounds the gap by $10^{-3}$ when $p$ is sufficiently large.
\end{remark}

\section{Constructing larger finite quotients}

Our next goal is to establish a suitable generalization of Theorem~\ref{thm:k-transitivity-base-field}  over larger fields $A = \F_q$. In that case, the situation is more complicated  in three different ways:
\begin{enumerate}[label=(\alph*)]
    \item  The action on non-zero vectors is not transitive;  there are several orbits, coming from intermediate fields  between $\F_p$ and $\F_q$.
    
    \item  The action of the Frobenius automorphism commutes with the $G$-action, so we cannot hope to have $k$-transitivity. We must consider the $G$-action on the quotient set modulo the Frobenius action.
    
    \item  The existence of  non-trival $E-1$ roots of unity  in the field causes significant technical complications.
\end{enumerate}

The rest of this section aims at addressing all those issues.

\subsection{Orbit invariants}\label{sec:orbits}

As before, we let $n \geq 3$, $G = G_{\F_p, \mathbf{e}}$ and $E=e_1\dots e_n$, and consider an $\F_p$-algebra $A$. 
Since $G$ preserves the grading of $R_n = \F_p[x_1,\dots,x_n]$ introduced in Definition~\ref{grading}, the   homogeneous components of the grading allow us to define orbit-invariants for the $G$-action on $A^n$. More precisely,   for each  $\phi \in \Hom(R_n, A)$  we  define 
the following subsets of $A$: $A_{\phi,s} = \phi(R_n^{(s)})$ and 
$A_{\phi} = \phi(R_n)$. For each $s$, the fibers of the map $\phi \mapsto A_{\phi,s}$ are $G$-invariant, since $A_{g\phi,s}  = A_{\phi,s}$ for all $g \in G$. Notice that  $A_{\phi,0}$  and $A_{\phi}$ are  subrings of $A$. Moreover, for each $s$ the subset $A_{\phi,s}$ is a module over $A_{\phi,0}$, and we have
$$
A_{\phi} = A_{\phi,0} + A_{\phi,1}  + \dots + A_{\phi,E-2}.$$
That sum is however not  direct in general.

\begin{lem}
\label{orbit:invariants}
Let $A$ be an algebraic field extension of $\F_p$,   and assume that $\phi$ is not the zero homomorphism.\footnote{We need to exclude the case when $\phi$ sends all generators to $0$, because in this case the sets $A_{\phi,i}$ consist only of $0$ and are not subfields, for $i \not =0$.} 
Then 
$A_{\phi,0}$ and $A_\phi$ are finite subfields of $A$. For each $s=0, 1, \dots, E-2$, the module  $A_{\phi,s }$ is a one-dimensional vector space over $A_{\phi,0}$.  

Furthermore, for any $\alpha_1 \in A_{\phi,1 } \setminus \{0\}$, we have that that $(\alpha_1)^s$ generates $A_{\phi,s }$ as a vector space over $A_{\phi,0}$, for all $s=1,\dots, E-2$.
In particular, we have  $[A_{\phi}:A_{\phi,0}]\leq E-1$.
\end{lem}
\begin{proof}
Both $A_{\phi,0}$ and $A_\phi$ are finitely generated as rings. Any non-zero finitely generated subring of $A$ is thus a finite field. The first claim follows.  Since $A_{\phi,s }$ is a module over $A_{\phi,0}$ it is a vector space. That vector space is $1$-dimensional since for any nonzero $\alpha \in A_{\phi, E-1-s}$, we have the inclusion $\alpha A_{\phi,s } \subset A_{\phi,0 }$,  and the multiplication by $\alpha$ is injective since $A_\phi$ is a field. Given any $\alpha_1 \in A_{\phi,1 } \setminus \{0\}$ and $s \in \{1, \dots, E-2$, we have  $\alpha_1^s \neq 0$, hence  $\alpha_1^s$ generates  the $1$-dimensional space  $A_{\phi,s }$   over $A_{\phi,0 }$.
\end{proof}

\begin{coro}
If $E-1$ is a prime number different from $p$,  then one of the following assertions holds.
\begin{enumerate}[label=(\alph*)]
    \item  $A_{\phi}=A_{\phi,s}$ for all $s$.
    \item  $A_{\phi,0}$ does contain all $(E-1)^\text{st}$ roots of $1$;  $A_{\phi}$ is an extension of $A_{\phi,0}$ of degree $E-1$ obtained by adding a $(E-1)^\text{st}$ root of some element in $A_{\phi,0}$.
    \item  $A_{\phi,0}$ does not contain any non-trivial $(E-1)^\text{st}$ root of $1$, and $A_{\phi}$ is obtained from $A_{\phi,0}$ by adjoining the $(E-1)^\text{st}$ roots of $1$. 
\end{enumerate}
\end{coro}
\begin{proof}
Fix $\alpha_1 \in A_{\phi,1 }$ and let $k$ be the smallest integer such that $\alpha_1^k \in A_{\phi,0 } $. Clearly $k$ divides $E-1$, thus either $k=1$ or $k=E-1$ since $E-1$ is prime. 

The case $k=1$ give $A_{\phi,s } = A_{\phi,0 }$, and  we are in (a). 

The case $k=E-1$ splits in two subcases.

Assume first that  $\alpha_1^{E-1} \not\in A_{\phi,0 }^{E-1}$. Then clearly $A_{\phi,0} \neq A_{\phi,0 }^{E-1}$. Therefore, the map $A_{\phi,0} \to A_{\phi,0 }^{E-1} : x \mapsto x^{E-1}$ is not surjective, hence it is not injective. Since its restriction to the non-zero elements is a group homomorphism, it follows that  $A_{\phi,0}$ contains a non-trivial  $(E-1)^\text{st}$ root of $1$, hence it contains all of them since $E-1$ is prime by hypothesis. This shows that (b) holds.  

It remains to treat the subcase where  $\alpha_1^{E-1} = \beta^{E-1}$ for some $\beta \in  A_{\phi,0 }$. It then follows that $\alpha'_1 = \alpha_1 \beta^{-1}$ is a $(E-1)^\text{st}$ root of $1$ contained in $A_{\phi,1 }$. Since $E-1$ is a prime, the subfield $A_{\phi,0}$ contains all $(E-1)^\text{st}$ roots of $1$ as soon as it contains any of them. Since $\alpha'_1  \not \in A_{\phi,0}$ in the case at hand, we deduce that    (c) holds. 
\end{proof}

Recall that the degree of $x_i$ with respect to the grading from Definition~\ref{grading} equals $ e_i\dots e_n$. %As before, we set $E = e_1\dots e_n$. 
The following easy observation will be useful. 

\begin{lem}  
\label{lem:A-phi-0}
Let $A$ be an algebraic field extension of $\F_p$,   and let $\phi \in \Hom(R_n, A)$ with $\phi(x_1) \neq 0$, where $R_n = \F_p[x_1,\dots,x_n]$. 
Then we have
$$
A_{\phi, 0} = \F_p\big(\phi(x_1)^{E-1}, \phi(x_1)^{-d_2}\phi(x_2),
\dots, 
\phi(x_1)^{-d_i}\phi(x_i),
\dots
%,\phi(x_1)^{-e_n}\phi(x_n)
\big),
$$
where  $d_i = e_i\dots e_n$ for all $i$.
\end{lem}
\begin{proof}
%To simplify the notation, we set  $\phi(i)= \phi(x_i)$ and $d_i = e_i\dots e_n$ for all $i$. 
Let 
$$
F = \F_p\big(\phi(x_1)^{E-1}, \phi(x_1)^{-d_2}\phi(x_2), \dots, \phi(x_1)^{-d_n}\phi(x_n)\big) \subseteq A_{\phi, 0}.
$$
By definition, the field $A_{\phi, 0}$ is the $\F_p$-algebra generated by the set 
$$
\left\{
\phi(x_{i_1})^{s_1}\dots \phi(x_{i_k})^{s_k} \bigg|\, s_j \in \mathbf Z, \ \sum_{j=1}^k d_{i_j} s_j \in (E-1)\mathbf Z
\right\}.
$$
Fix such a generator $\phi(x_{i_1})^{s_1}\dots \phi(x_{i_k})^{s_k}$ of $A_{\phi, 0}$. 
The field $F$ contains the element $(\phi(x_1)^{-d_{i_j}}\phi(x_{i_j}))^{s_j}$ for all $j$, hence also the product of all those elements, namely 
$$\phi(x_1)^{-\sum_{j=1}^k d_{i_j} s_j} \phi(x_{i_1})^{s_1}\dots \phi(x_{i_k})^{s_k}.$$ 
Since $\sum_{j=1}^k d_{i_j} s_j$ is a multiple of $E-1$ and since $F$ contains $\phi(x_1)^{E-1}$, we deduce that $F$ contains $\phi(x_{i_1})^{s_1}\dots \phi(x_{i_k})^{s_k}$. The result follows. 
\end{proof}

Given $\alpha \in A$, we define $\phi_\alpha \in \Hom(\F_p[x_1, \dots, x_n], A)$ by the assignments 
$$\phi_\alpha \colon \left\{
\begin{array}{rcl}
x_1 & \mapsto & \alpha\\
x_\ell & \mapsto & 0 \qquad \text{for all } \ell \geq 2.
\end{array}
\right.
$$
The next result shows that the maps $\phi \mapsto A_{\phi,0}$ and $\phi \mapsto A_{\phi,1}$ form a complete set of orbit invariants when $p \geq  E$. 

\begin{thm}
\label{thm:transitivity}
Let $A$ be an algebraic field extension of $\F_p$,   and let $\phi \in \Hom(R_n, A)$ be non-zero. Assume that $p \geq E$. 
\begin{enumerate}[label=(\roman*)]
    \item  There exists $\alpha \in A_{\phi,1}$ such that $\alpha^{E-1}$ generates the field $A_{\phi,0}$ over $\F_p$.
    
    \item For any $\alpha \in A_{\phi,1}$ such that $\alpha^{E-1}$ generates the field $A_{\phi,0}$ over $\F_p$, the homomorphism $\phi_\alpha$  
    belongs to the $G$-orbit of $\phi$.
\end{enumerate}
In particular, two non-zero points $\phi, \psi \in \Hom(R_n, A) \cong A^n$ are in the same $G$-orbit if and only if $(A_{\phi, 0}, A_{\phi, 1}) = (A_{\psi, 0}, A_{\psi, 1})$.
\end{thm}

\begin{proof}
Given any $\beta \in A_{\phi,1}$, we have $0 \neq \beta^{E-1} \in A_{\phi,0}$. Applying Lemma~\ref{lem:count} with $\gamma = \beta^{E-1}$ and $N = E- 1$ 
yields an element $\alpha' \in A_{\phi, 0}$ such that $A_{\phi, 0} = \F_p((\alpha' \beta)^{E-1})$. The assertion (i) holds by setting $\alpha = \alpha' \beta$.

We now focus on (ii). 
%As before, we write $\phi(i) = \phi(x_i)$. 
We denote the $G$-orbit of $\phi$ by $G\phi$. We also set $N = E-1$, and to lighten the notation, we slightly abuse notation by writing $\psi(i)$ instead of $\psi(x_i)$ for any $\psi \in \Hom(R_n, A)$. 

\setcounter{step}{0}
\begin{step}
For each $i \in \{1, \dots, n\}$ and each $\psi \in G\phi$, there exists $\chi \in G\phi$ such that $\left|\F_p(\chi(1)^N)\right| \geq \left|\F_p(\psi(i)^N)\right|$. 
\end{step}

Clearly, we may assume that $i \geq 2$. 
We invoke Lemma~\ref{claim:enlarge}(i) with exponent $k=t_{1, i} = e_1\dots e_{i-1}$, $\alpha = \psi(i)$ and $\beta = \psi(1)$. This ensures the existence of some $\lambda \in \F_p(\alpha^N)$ such that $\F_p((\beta + \lambda \alpha^k)^N)$ is at least as large as $\F_p(\alpha^N)$. Using Corollary~\ref{cor:PolynomialTransvection}, we construct a polynomial transvection $g \in G$ such that $g\psi(1) = \psi(1) + \lambda \psi(i)^k$ and $g\psi(\ell) = \psi(\ell)$ for all $\ell \geq 2$. We set $\chi = g \psi$, and we are done. 

\begin{step}
There exists $\psi \in G\phi$ such that $A_{\phi, 0} = \F_p(\psi(1)^N)$.
\end{step}

Let $\psi \in G\phi$ and $j \in \{1, \dots, n\}$ be such that  $\F_p(\psi(j)^N) \subseteq A_{\phi, 0}$ is a subfield of the largest possible cardinality. By the previous step, we may assume that $j=1$. In particular $\psi(1) \neq 0$.

Suppose now for a contradiction that $\F_p(\psi(1)^N)$ is strictly contained in $A_{\phi, 0}$. Then, by Lemma~\ref{lem:A-phi-0}, there exists an index $i\geq 2$ such that $\psi(1)^{-d_i}\psi(i) \not \in \F_p(\psi(1)^N)$. Therefore, we have  $\psi(i)^{N} \not \in \F_p(\psi(1)^N)$ or  $\psi(1)^{d_i}\psi(i)^{N-1} \not \in \F_p(\psi(1)^N)$. We then invoke Lemma~\ref{claim:enlarge}(ii) with exponent $k=t_{i, 1} = e_i\dots e_{n}$, $\alpha = \psi(1)$ and $\beta = \psi(i)$. This ensures the existence of some $\lambda \in \F_p(\alpha^N)$ such that $\F_p((\beta + \lambda \alpha^k)^N)$ is strictly larger than $\F_p(\alpha^N)$. As in the previous step, we may then find some $\chi \in G\phi$ with $\chi(i) = \beta + \lambda \alpha^k$. This contradicts the maximality property of $\psi$. 

\begin{step}
For each $\alpha \in A_{\phi,1}$ such that  $A_{\phi,0} = \F_p(\alpha^N)$, we have $\phi_\alpha \in G\phi$.
\end{step}

By the previous step, we may assume that $A_{\phi,0} = \F_p(\phi(x_1)^N)$. For all $j\not= 1$, we now apply Corollary~\ref{cor:PolynomialTransvection} to construct a polynomial transvection $g_j \in \prod_*\Trans j1* \subset G$ such that $g_j\phi(x_j) = \phi(x_j) + \lambda_j$ for any $\lambda_j \in A_{\phi,j}$. Using a suitable choice of $\lambda_j$ and taking a product over all $j$, we  obtain $g \in G$ such that 
$g\phi(1) = \phi(1)$, 
$g\phi(2) = \alpha$ and $g\phi(i) = 0$ for all $i \not= 1,2$. Applying Corollary~\ref{cor:PolynomialTransvection} a second time, we find $h$ such that $hg\phi(1) = \alpha$ and $hg\phi(i) = g\phi(i)$ for all $i\not = 2$. Finally, we do the same thing a third time to map $hg\phi$ to $\phi_\alpha$. This finishes the proof. 
\end{proof}

\begin{coro}\label{cor:Extension-prime-degree}
Assume that $p \geq E$. Let $q = p^\ell$, where  $\ell$ is a prime with $\ell \geq E$. Then the $G$-action on $\F_q^n$  has exactly three orbits, namely $\{(0,\dots,0)\}$, $\F_p^n \setminus \{(0,\dots,0)\}$ and $\F_q^n \setminus \F_p^n$.
\end{coro}
\begin{proof}
Under the assumptions,
the only subfields of $\F_q$ are $\F_p$ and $\F_q$. Given a non-zero  $\phi \in \F_q^n$, we have either $A_{\phi, 0} = \F_p$ or $A_{\phi, 0} = \F_q$.

In the former case, using Lemma~\ref{orbit:invariants} we obtain 
$$[A_\phi : \F_p] = [A_\phi : A] \leq E-1 < \ell = [\F_q : \F_p],$$
which implies that $A_\phi = \F_p$. Thus $\phi$ takes its values in $\F_p$, and $A_{\phi, s} = \F_p$ for all $s$. It follows from Theorem~\ref{thm:transitivity} that the $G$-orbit of $\phi$ is $\F_p^n \setminus \{(0,\dots,0)\}$ in this case.

In the latter case, we have $ A_{\phi, 0} = A_{\phi, s}$ for all $s$. By Theorem~\ref{thm:transitivity}, the $G$-orbit of  $\phi$ depends only on the pair  $( A_{\phi, 0}, A_{\phi, 1})$. It follows that all points of $\F_q^n \setminus \F_p^n$ belong to the same $G$-orbit.
\end{proof}

Lemma~\ref{lem:count} yields the following lower bound, showing the existence of one large orbit, which contains practically all points if  $p> 2E$.

\begin{coro}\label{cor:LargeOrbit}
Assume that $p \geq E$ and let $A$ be a finite field extension of $\F_p$ with $|A| = p^\ell$. All homomorphisms $\phi \in \Hom(R_n, A)$ such that $A_{\phi, 0} = A$ form a single $G$-orbit, whose cardinality is greater than 
$$p^{\ell n} \left( 1 -\frac{(E-1)^n}{p^n}\right).$$
\end{coro}
\begin{proof}
That the  homomorphisms $\phi$ with $A_{\phi, 0} = A$ form a single $G$-orbit follows directly from Theorem~\ref{thm:transitivity}. A sufficient condition ensuring that $A_{\phi, 0} = A$ is that $\phi(x_i)^{E-1}$ is not contained in a proper subfield of $A$ for some $i \in \{1, \dots, n\}$. The number of those $\phi$ such that $\phi(x_i)^{E-1}$ is  contained in a proper subfield of $A$ for each $i$ is at most $(E-1)^n\left( \sum_{d|\ell, d< \ell} p^d \right)^n$. It follows that the number of those $\phi$ such that $A_{\phi, 0} = A$ is at least
$
p^{\ell n} - (E-1)^n\left( \sum_{d|\ell, d< \ell} p^d \right)^n.
$
For $l\geq 4$, we have 
$\sum_{d|\ell, d< \ell} p^d \leq \sum_{ d\leq \ell/2} p^d \leq p^{n/2} \frac{p}{p-1}$, which implies
\begin{align*}
p^{\ell n} - (E-1)^n\left( \sum_{d|\ell, d< \ell} p^d \right)^n 
& \geq p^{\ell n} - p^{\ell n/2} \left(  (E-1) \frac{ p}{p-1}\right)^n \\
&> p^{\ell n} \left( 1 -\frac{(E-1)^n}{p^n}\right),
\end{align*}
as required, since $p^{-l/2} \frac{p}{p-1}  < 1/p$ for $l\geq 4$.
For $l=2,3$ one can directly verify that 
$$
p^{\ell n} - (E-1)^n\left( \sum_{d|\ell, d< \ell} p^d \right)^n = 
p^{\ell n} - (E-1)^n p^n > p^{\ell n} \left( 1 -\frac{(E-1)^n}{p^n}\right).
$$
\end{proof}

\subsection{Action of $\Z \ltimes \Z/(E-1)\Z$ commuting with $G$} 

Let $A$ be an algebraic field extension of $\F_p$. 
We  now describe an action of the semi-direct product 
$$\Gamma = \Z \ltimes \Z/(E-1)\Z$$ 
that commutes with the $G$-action. Let 
$$F\colon A \to A : a \mapsto a^p$$ 
denote the Frobenius automorphism. The group generated by $F$, acts by post-composition on $A^n \cong \Hom(\F_p[x_1, \dots, x_n], A)$, hence it commutes with the $G$-action. We therefore obtain a $\Z$-action by sending the generator $1$ to $F$. Since every subfield of $A$ is $F$-invariant, we deduce from Lemma~\ref{orbit:invariants} that for each $\phi \in \Hom(\F_p[x_1, \dots, x_n], A)$, the sets $A_\phi$ and $A_{\phi, 0}$ are invariant under $F$, while $A_{\phi, s}$ need not be for $s \geq 1$.

We now describe an action of the cyclic group $\Z/(E-1)\Z$ as follows.

\begin{lem}
\label{lem:cyclic-order-E-1}
Let $n \geq 3$ and $e_1, \dots, e_n \geq 1$ be integers  
and $A$ is a commutative unital $\F_p$-algebra. Let also  $\zeta \in A$ be 
an element such that $\zeta^{E-1} = 1$. 

Then the $G$-action on $A^n$ commutes with the  cyclic group of permutations generated by 
$$
m_\zeta \colon (a_1, \dots, a_n) \mapsto (\zeta a_1,\zeta^{E/e_1}a_2, \zeta^{E/e_1e_2}a_3, \dots , \zeta^{e_n} a_n).
$$
\end{lem}
\begin{proof}
Recall that $G$ is generated by $\tau_1(r), \dots, \tau_n(r)$, where $\tau_i(r) = \trans i{i+1}{e_i}r$. 
The automorphism $\tau_i(r)$ maps $(a_1, \dots, a_n)$ on $(a'_1, \dots, a'_n)$, where%
\footnote{We recall from   \S 7 that the action of $G$ on $A^n$ arises from the action of $G$ on $R_n = \F_p[x_1,\dots, x_n]$ and the natural map between  $\Hom(R_n,A)$ with $A^n$. We recall that, using that map, the transvection $\tau_i(r)$ acts as $x_i \mapsto x_i + rx_{i+1}^{e_i}$ on $R_n$, and by $a_i \mapsto a_i - ra_{i+1}^{e_i}$ on $A^n$.}
%, however since we act on the first argument of $\Hom$ one needs to switch between left and right actions, or equivalently act but the inverse of the element in $G$. This inverse leads to changing $+$ with $-$ in several formulas.}
$a'_i = a_i - r a_{i+1}^{e_i}$ and $a'_s = a_s$ for all $s \neq i$. To check that 
$\tau_i(r) m_\zeta = m_\zeta \tau_i(r)$, it suffices to consider the $i^\mathrm{th}$-coordinate. 
The $i^\mathrm{th}$-coordinate of $\tau_i(r) m_\zeta (a_1,\dots, a_n)$ equals 
$$
\zeta^{E/e_1 \dots e_{i-1}}a_i - r (\zeta^{E/e_1 \dots e_{i}}a_{i+1})^{e_i} = \zeta^{E/e_1 \dots e_{i-1}}(a_i - r a_{i+1}^{e_i}).
$$ 
The right-hand-side is the $i^\mathrm{th}$-coordinate of $m_\zeta \tau_i(r) (a_1,\dots, a_n)$, as required. 
\end{proof}

From now on and in the rest of this paper, we choose $\zeta \in A$ to be a generator of the cyclic group of $(E-1)$-roots of unity. We do not require that the multiplicative order of $\zeta$ equal $E-1$.
We obtain a $\Z/(E-1)\Z$-action by sending the generator $1$ on $m_\zeta$.  Since $F\circ m_\zeta \circ F^{-1} = m_{F(\zeta)}$, we indeed obtain an action of the semi-direct product 
$\Gamma = \Z \ltimes \Z/(E-1)\Z$  
that commutes with the $G$-action. That action need not be faithful: if $A$ is finite the Frobenius automorpisms is of finite order, so the group $\Z$ acts via a proper quotient; similarly if the order of $\zeta$ is less than $E-1$ (this happens if $p$ divides $E-1$, even if $A$ is algebraically closed), the group $\Z/(E-1)\Z$ also acts via a proper quotient.

The orbit invariant $A_{\phi,0}$ is  preserved by the $\Gamma$-action. However $A_{\phi,1}$ and $A_{\phi}$ are not in general.

The following observation is useful to distinguish the $\Gamma$-orbits on the affine space $A^n \cong \Hom(\F_p[x_1, \dots, x_n], A)$. As before, 
%we let $\ell$ denote the degree of $A$ over $\F_p$. Moreover, given $\alpha \in A$, 
we define $\phi_\alpha \in \Hom(\F_p[x_1, \dots, x_n], A)$ by $x_1 \mapsto \alpha$ and $x_s \mapsto 0$ for all $s \geq 2.$ 

\begin{lemma}\label{lem:Gamma-orbits}
Let $A$ 
be an algebraic field extension of $\F_p$  and let $\alpha_1, \dots, \alpha_s \in A$.

\begin{enumerate}[label=(\roman*)]
    \item The homomorphisms $\phi_{\alpha_1}, \dots, \phi_{\alpha_s}$ lie in pairwise distinct $\Gamma$-orbits if and only if the minimal polynomials of $\alpha_1^{E-1}, \dots, \alpha_s^{E-1}$ over $\F_p$ are pairwise distinct. 
    
    \item Let $L \subseteq A$ be a finite subfield of degree $\ell$ over $\F_p$.    Assume that $\F_p(\alpha_i^{E-1})  =L$ for all $i=1, \dots, s$, and that  the minimal polynomials of $\alpha_1^{E-1}, \dots, \alpha_s^{E-1}$ over $\F_p$ are pairwise distinct.  Let  $k \geq s+1$ be an integer. Assume  that
    $$ 
    k \leq \frac{p^{\ell-1} (p-E)}{\ell  E}
    $$
    if $L \neq \F_p$, and that $k\leq \frac{p-1}{E-1}$ if $L = \F_p$. 
    Then there exist elements $\alpha_{s+1}, \dots , \alpha_k \in L$ such that $\F_p(\alpha_i^{E-1}) = L$ for all $i$, and the minimal polynomials of $\alpha_1^{E-1}, \dots, \alpha_k^{E-1}$ over $\F_p$ are pairwise distinct.
\end{enumerate}

\end{lemma}

\begin{proof}
Two elements of $A$ have the same minimal polynomial over $\F_p$ if and only if they belong to the same $\langle F \rangle$-orbit. Let $\alpha , \beta \in A$. Then $\phi_\alpha$ and $\phi_{\beta}$ are in the same $\Gamma$-orbit if and only if there exist integers $i \in \{0, 1, \dots, E-1\}$ and $j \in \mathbf N$  such that $\beta = (\zeta^i \alpha)^{p^j}$. 

If the latter holds, then we have $\beta^N = (\alpha^N)^{p_j}$, where $N=E-1$. Hence $\alpha^N$ and $\beta^N$ have the same minimal polynomial over $\F_p$. Conversely, if $\beta^N = (\alpha^N)^{p_j}$, then $\beta^{-1}\alpha^{p_j} \in A$ is an $N^{\text{th}}$ root of unity, so there exists an integer  $m$ with $\beta^{-1}\alpha^{p_j} = \zeta^m$. Since the multiplicative group of $N^{\text{th}}$ roots of unity is $\langle F\rangle$-invariant, this implies that there exists $i \in \mathbf Z$ with $\beta = (\zeta^i \alpha)^{p^j}$. The assertion (i) follows. 

For the assertion (ii), we use a similar counting argument as in the proof of Lemma~\ref{lem:count}. The number of elements $\beta \in L$ is such that $\F_p(\beta^N) \neq L$ is at most $N \sum_{t=1}^{\ell-1} p^t = N \frac{p^\ell-p}{p-1}$ since $|L| = p^\ell$. Hence the complement of that set, that we denote by $L_0$, has cardinality at least $\frac{p^\ell(p-E)+p(E-1)}{p-1}$. The group $\Gamma$ acts on $A$ by permutations via the map $\alpha \mapsto \phi_\alpha$. 
As observed at the beginning of the proof, two elements $\alpha, \beta \in L$ are in the same $\Gamma$-orbit if and only if there exist integers $i \in \{0, 1, \dots, E-2\}$  and $j \in \mathbf N$  such that $\beta = (\zeta^i \alpha)^{p^j}$. Since $\zeta$ is a generator of the subgroup of $(E-1)^\mathrm{st}$-roots of $1$ in $A$, the subfield $\F_p(\zeta)$ is invariant under the Frobenius automorphism. Therefore we have  $\beta = \zeta^s \alpha^{p^j}$ for some $s \in \{0, 1, \dots, E-2\}$. Since $\alpha$ belongs to $L$, which has order~$p^\ell$, we may take $j \leq \ell-1$. It follows that the $\Gamma$-orbit of any element $\alpha \in L$ is of size at most $\ell(E-1)$. 
Hence the elements of $L_0$ fall into at least $K$ distinct $\Gamma$-orbits, where
    $$K =  \frac{p^\ell(p-E) + p(E-1)}{\ell(E-1)(p-1)}.$$
Clearly, we have $\frac{p^{\ell-1} (p-E)}{\ell  E} = \frac{p^\ell (p-E)}{\ell p E} \leq K$. The assertion (ii) follows in case $L \neq \F_p$. 

If $L =\F_p$ (i.e. $\ell=1$), the argument above simplifies. Indeed, for each non-zero $\beta \in L$, we have $\F_p(\beta^N) = \F_p$, so that $L_0= \F_p\setminus \{0\}$ has cardinality $p-1$ in this case. The $\Gamma$-orbit of any element $\alpha \in \F_p$ is of size at most $E-1$, so    the elements of $L_0$  fall into at least $\frac{p-1}{E-1}$ distinct $\Gamma$-orbits. The assertion (ii) follows since $k\leq \frac{p-1}{E-1}$ by  hypothesis. 
\end{proof}

\subsection{Higher transitivity}

Throughout this section, we assume that $n \geq 3$. 
The $G$-orbits on $A^n$ are described by Theorem~\ref{thm:transitivity}. Moreover Theorem~\ref{thm:k-transitivity-base-field} shows that, under suitable assumptions on $p$ and $E$, the $G$-action on one of the orbits, namely $\mathbf F_p^n \setminus\{0\}$,  is $k$-transitive for all sufficiently small $k$. Our next goal is to show that the $G$-action is almost $k$-transitive on each $G$-orbit. The obstruction to being $k$-transitive in the strict sense comes from the $\Gamma$-action, that commutes with the $G$-action. The following theorem shows that this is the only obstruction: on each $G$-orbit, the $G$-action is $k$-transitive on the blocks of imprimitivity formed by the $\Gamma$-orbits. 

As before, we let $R_n =\F_p[x_1, \dots, x_n]$, $\mathbf e = (e_1, \dots, e_n)$, $G = G_{\F_p, \mathbf e }$ and  $E=e_1\dots e_n$. Given a field extension $A$ of $\F_p$ and  $\phi \in \Hom(R_n, A)$, we use the notation $A_{\phi, s}$ from \S\ref{sec:orbits}. 

\begin{thm}
\label{thm:almost-k-transitivity}
Let $\F_p \subseteq L \subseteq A$ be finite field extensions of $\F_p$, and let   $\ell = [L:\F_p]$. We assume that 
$$p \geq 3E-2.$$
Let $k$ be an integer, %
and assume that $k \leq \frac{p^{\ell-1} (p-E)}{\ell  E}$ if $L \neq \F_p$,  and that $k \leq \frac{p-1}{E-1}$ if $L= \F_p$. Let also $\phi_1, \dots, \phi_k \in \Hom(R_n, A) \cong A^n$ be homomomorphisms such that $A_{\phi_i, 0} = L$ for all $i$. For each $i$, let also $\alpha_i \in A_{\phi_i, 1}$ be such  that $L = \F_p(\alpha_i^{E-1})$. 

Assume that $\phi_1, \dots, \phi_k$ (resp. $\phi_{\alpha_1}, \dots, \phi_{\alpha_k}$) belong to pairwise distinct $\Gamma$-orbits. 
Then there exists an element $g \in G$ and, for each $i=1, \dots, k$, an element $\gamma_i \in \Gamma$ such that 
$$
g(\gamma_i(\phi_i)) = \phi_{\alpha_i}
\,\, \mbox{for all} \,\, i=1, \dots, k.
$$
\end{thm}

\begin{proof}
%Without loss of generality we can assume that $\F_p$ is a proper subfield of $L$, otherwise we can simply use Theorem~\ref{thm:k-transitivity-base-field}. 
We proceed by induction on $k$. The base case $k=1$ is afforded by Theorem~\ref{thm:transitivity}. Assume henceforth that $k \geq 2$ and that the required conclusion is true for $k-1$. Using the induction hypothesis, we may assume without loss of generality that $\phi_i = \phi_{\alpha_i}$ for each $i=1, \dots, k-1$. It remains to show that there exists $g \in G$ and $\gamma \in \Gamma$ such that $g(\phi_{\alpha_i}) = \phi_{\alpha_i}$ for all $i \leq k-1$, and $g(\gamma(\phi_k)) = \phi_{\alpha_k}$. To construct  those elements, we proceed in several steps, which essentially follow those taken in the proof of Theorem~\ref{thm:transitivity}.  
%We write $\psi(i) = \psi(x_i)$ for all $\psi \in \Hom(R_n, A)$, and set $N=E-1$. 
As before set $N=E-1$ and, to lighten the notation, we slightly abuse notation by writing $\psi(i)$ instead of $\psi(x_i)$ for any $\psi \in \Hom(R_n, A)$.

\setcounter{step}{0}
\begin{step}\label{step:main:1}
For each $j \in \{2, \dots, n\}$, there exists $h \in G$ such that 
$h(\phi_{\alpha_i}) = \phi_{\alpha_i}$ for all $i \leq k-1$, and that 
$\left|\F_p(h\phi_k(1)^N)\right| \geq \left|\F_p(\phi_k(j)^N)\right|$. 
\end{step}

We invoke Lemma~\ref{claim:enlarge}(i) with  $\alpha = \phi_k(j)$ and $\beta = \phi_k(1)$. This ensures the existence of some $\lambda \in \F_p(\alpha^N)$ such that $\F_p\left((\beta + \lambda \alpha^{t_{1, j}})^N\right)$ is at least as large as $\F_p(\alpha^N)$. We can write $\lambda = P(\alpha^N)$ for some polynomial $P \in \F_p[y]$. Using Corollary~\ref{cor:PolynomialTransvection}, we construct a polynomial transvection $h \in G$ such that for all $\psi \in \Hom(R_n, A)$, we have 
$h\psi(1) = \psi(1) +  \psi(j)^{t_{1, j}}P(\psi(j)^N)$  and $h\psi(m) = \psi(m)$ for all $m \geq 2$. In particular $h\phi_i(1) = \phi_i(1) + 0 = \alpha_i$ for all $i\in \{1, \dots, k-1\}$, and $h\phi_k(1) = \beta + \alpha^{t_{1, j}} \lambda$. The claim follows. 

\begin{step}\label{step:main:2}
Suppose that  $\left|\F_p(\phi_k(1)^N)\right| \geq \left|\F_p(\phi_k(j)^N)\right|$ for all $j$. If $\F_p(\phi_k(1)^N) \neq L$, then there exists $h \in G$ and $j \in \{2, \dots, n\}$ such that $h(\phi_{\alpha_i}) = \phi_{\alpha_i}$ for all $i \leq k-1$ and  
$\left|\F_p(h\phi_k(j)^N)\right| > \left|\F_p(\phi_k(1)^N)\right|$. 
\end{step}

Assume first that $\phi_k(1) = 0$, so that $\F_p(\phi_k(1)^N) = \F_p$ and hence $\F_p(\phi_k(j)^N) = \F_p$ for all $j$ by the assumptions made in this step. By hypothesis, we jave $A_{\phi_i, 0}=L$ for all $i$. In particular $\phi_k$ is non-zero, hence there is some $m$ such that $\phi_k(m) \neq 0$. We invoke Corollary~\ref{cor:PolynomialTransvection} to find a polynomial transvection $h_0 \in G$ such that  $h_0\psi(1 $ and $ h_0\psi(m) = \psi(m)$ for all $m \geq 2$. Hence $h_0 \phi_i = \phi_i$ for all $i=1, \dots, k-1$ and $h_0\phi_k(1) \neq 0$. We set $\phi'_k = h_0 \phi_k$. Since $\F_p(\phi'_k(j)^N) = \F_p(\phi_k(j)^N = \F_p$ for all $j \geq 2$, we have $\left|\F_p(\phi'_k(1)^N)\right| \geq \left|\F_p(\phi'_k(j)^N)\right|$ for all $j$. Moreover $\F_p(\phi'_k(1)^N) \subseteq L$ since $\phi_k$ and $\phi'_k$ are in the same $G$-orbit. Therefore, upon replacing $\phi_k$ by $\phi'_k$, we may and will assume henceforth that $\phi_k(1) \neq 0$. 

By Lemma~\ref{lem:A-phi-0}, there exists an index $j\geq 2$ such that $\phi_k(1)^{-d_j}\phi_k(j) \not \in \F_p\left(\phi_k(1)^N\right)$, where $d_j = e_j\dots e_n$. Therefore, we have  $\phi_k(j)^{N} \not \in \F_p\left(\phi_k(1)^N\right)$ or  $\phi_k(1)^{d_j}\phi_k(j)^{N-1} \not \in \F_p\left(\phi_k(1)^N\right)$. We then invoke Lemma~\ref{claim:enlarge}(ii) with exponent $t_{j, 1} = d_j$, $\alpha = \phi_k(1)$ and $\beta = \phi_k(j)$. This ensures the existence of some $\lambda \in \F_p(\alpha^N)$ such that $\F_p\left((\beta + \lambda \alpha^{d_j})^N\right)$ is strictly larger than $\F_p\left(\alpha^N\right)$.

By hypothesis, the elements $\phi_{\alpha_1}, \dots, \phi_{\alpha_k}$ are in pairwise distinct $\Gamma$-orbits,
%. By Lemma~\ref{lem:Gamma-orbits}, 
this means that $\alpha_1^N, \dots, \alpha_k^N$ have distinct minimal polynomials over $\F_p$. By hypothesis, we have $L = A_{\phi_i, 0}$ for all $i$. In view of Lemma~\ref{lem:A-phi-0}, this implies that   $\F_p\left(\alpha_i^N\right) = L$ for all $i\in \{1, \dots, k-1\}$, while $\F_p\left(\phi_k(1)^N\right) \neq L$ by assumption. Therefore we have that  $\alpha_1^N, \dots, \alpha_{k-1}^N, \phi_k(1)^N$ have distinct minimal polynomials over $\F_p$. Therefore, by Lemma~\ref{lem:polynomial-map-k-tuple}, there is a polynomial $P \in \F_p[y]$ such that $P(\alpha_i^N) = 0$ for all $i=1, \dots, k-1$ and $P(\phi_k(1)^N)= P(\alpha^N) = \lambda$. 

Finally, we invoke Corollary~\ref{cor:PolynomialTransvection} to find a polynomial transvection $h \in G$ such that $h\psi(j) =\psi(j) + \psi(1)^{d_j}P(\psi(1)^N)$ and $h\psi(m) = \psi(m)$ for all $m \neq j$. The claim follows.

\begin{step}\label{step:main:3}
There exists $h \in G$ such that $h(\phi_{\alpha_i}) = \phi_{\alpha_i}$ for all $i \leq k-1$, and that 
 $\F_p(h\phi_k(1)^N) = L$.
\end{step}

By applying iteratively the first two steps, we find a sequence of elements $h_m \in G$ which all fix $\phi_{\alpha_i}$ for $i \leq k-1$, and such that the cardinality of  $\F_p\left(h_m\phi_k(1)^N\right)$ strictly increases with $m$. Recall moreover that $g\phi_k(1)^N \in A_{\phi_k, 0} = L$ for all $g \in G$, so that $\F_p\left(h_m\phi_k(1)^N\right) \subseteq L$ for all $m$. Therefore, the process stops once an element $h_m$ satisfying the condition  that $\F_p\left(h_m\phi_k(1)^N\right) = L$ is found. This proves the claim. 

\begin{step}\label{step:main:4}
Suppose that $\F_p(\phi_k(1)^N) = L$. Then there exists $j \in \{1, \dots, n\}$, $\gamma \in \Gamma$ and $h \in G$ such that the elements $h\phi_1(j)^N, h\phi_2(j)^N, \dots h\phi_{k-1}(j)^N, h\gamma \phi_k(j)^N$ have pairwise different minimal polynomials, and each of them generates $L$ over $\F_p$. 
\end{step}

Set $\beta = \phi_k(1)$. 
Suppose first that $\alpha_1^N, \dots, \alpha_{k-1}^N, \beta^N$ have pairwise different minimal polynomials. Recall that $\phi_i = \phi_{\alpha_i}$ for all $i \leq k-1$. Therefore, the required conclusion is satisfied with $j=1$, $\gamma= 1$ and $h=1$ in this case.

Suppose next that $\phi_k(m) = 0$ for all $m \geq 2$. This means that $\phi_k = \phi_\beta$.  Since $\phi_1, \dots, \phi_k$ lie in pairwise distinct $\Gamma$-orbits, we deduce from Lemma~\ref{lem:Gamma-orbits} that $\alpha_1^N, \dots, \alpha_{k-1}^N, \beta^N$ have pairwise different minimal polynomials. Therefore, we are reduced to the first case above.

We assume henceforth that there is $s \in \{1, \dots, k-1\}$ such that $\alpha_s^N$ and $\beta^N$ have the same minimal polynomials. By the previous paragraph, this implies that $\phi_k(j) \neq  0$ for some $j \geq 2$. Upon replacing $\phi_k$ by $\gamma\phi_k$ for a suitable $\gamma \in \Gamma$, we may then assume that 
$$\phi_s(1) = \alpha_s = \beta = \phi_k(1).$$ 

We next claim that there exists some $\delta \in L$ such that $\delta^N$ and $(\phi_k(j) + \delta)^N$ have distinct minimal polynomials, and that each of these two elements generates $L$ as an $\F_p$-algebra. We verify this   by a counting argument as in Lemma~\ref{claim:enlarge}, as follows. 

If $\F_p(\delta^N)$ is a proper subfield of $L$, then $\delta$ satisfies a polynomial equation over $\F_p$ of degree $Np^t$ for some $t \in \{1, \dots, \ell-1\}$. The same conclusion holds if $\F_p\big((\phi_k(j) + \delta)^N \big)$ is a proper subfield of $L$. Finally, if  $\delta^N$ and $(\phi_k(j) + \delta)^N$ have the same minimal polynomial, then we have $(\phi_k(j) + \delta)^N = \delta^{Np^t}$ for some $t\in \{0, 1, \dots,  \ell-1\}$.  This means that
either $\delta$ satisfies an equation of degree $N-1$, or $\delta$ satisfies an equation of degree $Np^t$ for some $t \in \{1, \dots, \ell-1\}$. Therefore, the set of those $\delta$ that must be excluded is of cardinality at most the sum of the degrees of these equations 
$$
N-1 + 3N \sum_{t=1}^{\ell-1} p^t = N-1 + 3N \frac{p^\ell-p}{p-1}.
$$
Using the hypothesis that $p \geq 3E-2 = 3N+1$, we have 
$$(3N+1) p^\ell \leq p^{\ell+1} < p^{\ell+1} + 2Np + p + N-1.$$
This implies that
$$N-1 + 3N \frac{p^\ell-p}{p-1} < p^\ell,$$
so that the field $L$, which is of order $p^\ell$ contains at least one element $\delta$ outside of that critical set. 

We fix that element $\delta$ and we set $\beta_s = \delta$ and $\beta_k = \phi_k(j) + \delta$. 
Using the hypothesis that $k  \leq \frac{p^\ell (p-E)}{\ell p E}$ if $\ell \geq 2$, and that $k \leq \frac{p-1}{E-1}$ if $\ell = 1$, we deduce from Lemma~\ref{lem:Gamma-orbits}(ii) that there exist $\beta_1, \dots, \beta_{s-1}, \beta_{s+1}, \dots, \beta_{k-1} \in L$ such that $\beta_1^N, \dots, \beta_k^N$ are pairwise distinct minimal polynomials, and $\F_p(\beta_m^N) = L$ for each $m$. 

Since $\alpha_1^N, \dots, \alpha_{k-1}^N$ have pairwise different minimal polynomials, Lemma~\ref{lem:polynomial-map-k-tuple} affords a polynomial $P\in \F_p[y]$ such that 
$$
P(\alpha_m^N) = \frac{\beta_m}{\alpha_m^{d_j}} \qquad \text{for each } m=1, \dots, k-1,
$$ 
where $d_j =t_{j, 1} = e_j \dots e_n$. 
Using Corollary~\ref{cor:PolynomialTransvection}, we obtain a polynomial transvection $h \in G$ such that for each $\psi \in \Hom(R_n, A)$, we have $$h\psi(j) = \psi(j)+\psi(1)^{d_j}P(\psi(1)^N)$$ 
and $h\psi(m) = \psi(m)$ for all $m \neq j$. It follows that $h\phi_i(j) = \beta_i$ for all $i=1, \dots, k$, and the step is complete. 

\begin{step}
End of the proof.
\end{step}

In view of Step~\ref{step:main:3}, we may assume that $\F_p(\phi_k(1)^N) = L$. We then invoke Step~\ref{step:main:4}, and use the symbols $j, h, \gamma$ arising from its statement. Set $\chi_i = h\phi_i$ for $i \leq k-1$, and $\chi_k = h\gamma\phi_k$. Since $n \geq 3$, we may choose an index $l \in \{1, \dots, n\}\setminus \{1, j\}$. By Step~\ref{step:main:4}, we know that $\chi_1(j)^N, \chi_2(j)^N, \dots, \chi_k(j)^N$ have pairwise different minimal polynomials, and each of those elements generate $L$ over $\F_p$. Therefore, we may apply Lemma~\ref{lem:polynomial-map-k-tuple} and Corollary~\ref{cor:PolynomialTransvection}, in order to construct  a polynomial  transvection $h' \in G$ such that $h'\chi_i(l) = \alpha_i$ for all $i$. Repeating the same argument, we obtain another polynomial  transvection $h'' \in G$ such that $h''h'\chi_i(1) = \alpha_i$ for all $i$ (obviously, we may take $h''$ trivial if $l=1$). Finally, we invoke again Lemma~\ref{lem:polynomial-map-k-tuple} and Corollary~\ref{cor:PolynomialTransvection} to find $n-1$ further polynomial transvections $h_2, \dots, h_n \in G$, commuting pairwise, such that $h_m h'' h'\chi_i(m) = 0$  and $h_m h'' h'\chi_i(1) = \alpha_i$ for all $i$. It follows that $h_2 h_3 \dots h_n h'' h'\chi_i = \phi_{\alpha_i}$ for all $i=1, \dots, k$. Therefore, the element $g = h_2 h_3 \dots h_n h'' h' h$ has the desired property. This finishes the proof.
\end{proof}

Specializing to the large orbit described in Corollary~\ref{cor:LargeOrbit}, and remembering the a finite $k$-transitive group in degree~$\geq 25$ contains the full alternating group as soon as $k \geq 4$, we obtain the following. 

\begin{coro}\label{cor:Alt-quotient}
Let $A$ be a finite field extension of $\F_p$ of degree $\ell$. Let $\Phi$ denote the set of those homomorphisms $\phi \in \Hom(R_n, A)$ such that $A_{\phi, 0} = A$. 

If $E \geq 2$,  $p \geq 3E-2$ and $p^{\ell-1} \geq 4\ell$, 
then the image of $G$ in $\Sym(\Gamma\backslash \Phi)$ induced by the $G$-action on the  $\Gamma$-orbits on $\Phi$ induces the full alternating group $\Alt(\Gamma\backslash \Phi)$.
\end{coro}
\begin{proof}
Since $E\geq 2$ and $p\geq 3E-2$, we have $p-E \geq 2E-2 \geq E$. Since $p^{\ell-1} \geq 4\ell $, we infer that $p^{\ell-1}(p-E) \geq 4\ell  E $, so that $\frac{p^{\ell-1} (p-E)}{\ell  E} \geq 4$. In view of Theorem~\ref{thm:almost-k-transitivity}, it follows that the $G$-action on $\Gamma\backslash \Phi$ is $4$-transitive. Moreover, setting $N = E-1$, we have 
$$
|\Gamma\backslash\Phi| \geq \frac{p^{n\ell}}{\ell N}\left(1-\left(\frac{N}{p} \right)^n \right)
$$
by Corollary~\ref{cor:LargeOrbit}, since the orbits of $\Gamma$ have size at most $\ell N$. 

By hypothesis, we have $p \geq 3N+1 > 3N$. Moreover $n \geq 3$ and $\ell \geq 1$, so $\frac{p^{n\ell}}{\ell N} \geq \frac{p^3} N = p^2 \frac p N$. We deduce that 
$$
|\Gamma\backslash\Phi| \geq p^2 \left(\frac{p}{N} -\left(\frac{N}{p}\right)^{n-1} \right) > p^2 \left(3- \frac{1}{9}\right)> 25.
$$
The conclusion follows using \cite[Th.~4.11]{Cameron}, together with the fact that  
$G$  does not have any quotient isomorphic to $\Sym(n)$ for $n \geq 2$ since it    is generated by elements of odd order.
\end{proof}

\begin{remark}
If the hypothesis that $p^{\ell-1} \geq 4\ell$ is replaced by $p^{\ell-2} \geq \ell$, then the same argument using Theorem~\ref{thm:almost-k-transitivity} shows that the $G$-action on $\Gamma \backslash \Phi$ is $p$-transitive. Similarly as in Remark~\ref{rem:CFSG-free}, the conclusion that $G$ induces the full alternating group on $\Gamma\backslash \Phi$ can then be established using \cite[Corollary to Theorem~A]{Pyber} (and thereby avoiding  the CFSG), provided $p$ is large enough. 
\end{remark}

The following consequence, which is stated as Theorem~\ref{thm:Alt-quotient-intro} in the introduction, is immediate. 

\begin{coro}\label{cor:Alt-quotient:bis}
If $E \geq 2$ and $p \geq 3E-2$, then the group $G$ has a quotient isomorphic to $\Alt(d)$ for infinitely many degrees $d$.
\end{coro}

In the special case where $E=2$, we can specify more explicitly some values of $d$. 

\begin{coro}\label{cor:Alt-quotient:ter}
Suppose that $E=2$ and  $p \geq 5$. Then for each odd   prime $\ell$, the group  $G$ has a quotient isomorphic to $\Alt\left(\frac{p^{\ell n} -p^n}{\ell}\right)$.
\end{coro}
\begin{proof}
The hypotheses imply that $\Phi = A^n \setminus (\F_p)^n$, see Corollary~\ref{cor:Extension-prime-degree}. Moreover each $\Gamma$-orbit on $\Phi$ has size~$\ell$, since the Frobenius automorphism acts on $\Phi$ without fixed points. Since the map $x \mapsto  x - \log_p(x)$ is strictly increasing for $x \geq 1$, we have $\ell - 2 \geq \log_p(\ell)$ since $\ell \geq 3$ and $p \geq 5$, hence $p^{\ell-1} \geq p \ell \geq 4\ell$.  Thus the conclusion  follows from Corollary~\ref{cor:Alt-quotient}.
\end{proof}

By definition, the group $G = G_{\F_p, \mathbf e}$, with $\mathbf e=( e_1, \dots, e_n)$, is $n$-generated. In the case where $e_1= e_2 = \dots = e_n$, the group $G$ is normalized by the  automorphism $\sigma  \in \Aut(\F_p[x_1, \dots, x_n])$ that  permutes cyclically the indeterminates. We assume that $e_1 > 1$, so that the automorphism $\sigma$ does not preserve the grading. Hence $G$ is a proper subgroup of $\widetilde G =\langle \sigma \rangle  G $ which is isomorphic to the semi-direct product $ \langle \sigma \rangle  \ltimes G$. Clearly  $\widetilde G$ is generated by the pair $\{\sigma, \alpha_{1, 2}^{(e_1)}(1)\}$. Therefore, the following corollary provides an infinite family of Cayley graphs of degree $4$ for alternating groups, that form expanders by Theorem~\ref{thm:T-for-G}.

\begin{coro}\label{cor:2-generated}
Let $p \geq 23$ be a prime such that $p \neq 1 \mod 7$ and set $\widetilde G = \langle \sigma \rangle \ltimes G_{\F_p,  (2, 2, 2)}$. Then for each  prime $\ell \geq 5$, the group  $\widetilde G$ has a quotient isomorphic to $\Alt\left(\frac{p^{3\ell } -p^3}{\ell}\right)$.
\end{coro}
\begin{proof}
Let $G = G_{\F_p, (2, 2, 2)}$. The hypotheses imply that $\Phi = A^n \setminus (\F_p)^n$, see Corollary~\ref{cor:Extension-prime-degree}. Since $G$ is normal in $\widetilde G$, it follows that $\Phi$ is $\sigma$-invariant. 

Notice that $p^\ell - 1$ is relatively prime to $E-1 = 7$. Indeed,  the equality $p^\ell =1 \mod 7$ implies that the multiplicative order of $p$ modulo $7$ divides the gcd of $6$ and $ \ell$. Since $\ell$ is prime, we obtain three cases: $p = 1 \mod 7$, or $\ell =2$, or $\ell = 3$.  
None of those cases occurs in view of the hypotheses on $p$ and $\ell$. 

It follows that the $\Gamma$-orbits on $\Phi$ all have size $\ell$. Therefore the $G$-action on $\Gamma\backslash \Phi$ has the full alternating group $\Alt(\frac{p^{3\ell } -p^3}{\ell})$ as its image, see Corollary~\ref{cor:Alt-quotient}. Since $\sigma$ is of order $3$, it acts on  $\Gamma\backslash \Phi$ as an even permutation. The conclusion follows.
\end{proof}

Corollary~\ref{coro:Alt-intro} follows from Corollaries~\ref{cor:Alt-quotient:ter} and~\ref{cor:2-generated}.

%\bibliographystyle{amsalpha}
%\bibliographystyle{plain}
%\bibliography{biblio}

\end{document}